\documentclass{article}
\usepackage{amsmath}
\usepackage{amsfonts}
\usepackage{cite}
\usepackage{hyperref}
\usepackage{graphicx}

\newtheorem{theorem}{Theorem}

\newtheorem{corollary}[theorem]{Corollary}

\newtheorem{definition}[theorem]{Definition}

\newtheorem{lemma}[theorem]{Lemma}

\newtheorem{remark}[theorem]{Remark}

\newenvironment{proof}[1][Proof]{\noindent\textbf{#1.} }{\ \rule{0.5em}{0.5em}}

\begin{document}

\title{Part 1. Infinite series and logarithmic integrals associated to
differentiation with respect to parameters of the Whittaker $\mathrm{M}%
_{\kappa ,\mu }\left( x\right) $ function.}
\author{Alexander Apelblat$^{1}$, Juan Luis Gonz\'{a}lez-Santander$^{2}$. \\
$^{1}$ Department of Chemical Engineering, \\
Ben Gurion University of the
Negev, \\
84105 Beer Sheva, 84105, Israel. apelblat@bgu.ac.il\\
$^{2}$ Department of Mathematics, Universidad de Oviedo, \\
33007 Oviedo, Spain. gonzalezmarjuan@uniovi.es}
\maketitle

\begin{abstract}
First derivatives of the Whittaker function $\mathrm{M}_{\kappa ,\mu }\left(
x\right) $ with respect to the parameters are calculated. Using the
confluent hypergeometric function, these derivarives can be expressed as
infinite sums of quotients of the digamma and gamma functions. Also, it is
possible to obtain these parameter derivatives in terms of finite and
infinite integrals with integrands containing elementary functions (products
of algebraic, exponential and logarithmic functions) from the integral
representation of $\mathrm{M}_{\kappa ,\mu }\left( x\right) $. These
infinite sums and integrals can be expressed in closed-form for particular
values of the parameters. For this purpose, we have obtained the parameter
derivative of the incomplete gamma function in closed-form. As an
application, reduction formulas for parameter derivatives of the confluent
hypergeometric function has been derived, as well as some finite and
infinite integrals containing products of algebraic, exponential,
logarithmic and Bessel functions. Finally, some reduction formulas for the
Whittaker functions $\mathrm{M}_{\kappa ,\mu }\left( x\right) $ and integral
Whittaker functions $\mathrm{Mi}_{\kappa ,\mu }\left( x\right) $ and $%
\mathrm{mi}_{\kappa ,\mu }\left( x\right) $ are calculated.
\end{abstract}


\textbf{AMS\ Subject Classification}:\ 33B15, 33B20, 33C10, 33C15, 33C20,
33C50, 33E20.

\section{Introduction}

Introduced in 1903 by Whittaker \cite{whittaker1903expression}, the $\mathrm{%
M}_{\kappa ,\mu }\left( x\right) $ and $\mathrm{W}_{\kappa ,\mu }\left(
x\right) $ functions are defined as:\
\begin{eqnarray}
\mathrm{M}_{\kappa ,\mu }\left( x\right) &=&x^{\mu
+1/2}e^{-x/2}\,_{1}F_{1}\left( \left.
\begin{array}{c}
\frac{1}{2}+\mu -\kappa \\
1+2\mu%
\end{array}%
\right\vert x\right)  \label{M_k,mu_def} \\
2\mu &\neq &-1,-2,\ldots  \notag
\end{eqnarray}%
and%
\begin{equation}
\mathrm{W}_{\kappa ,\mu }\left( x\right) =\frac{\Gamma \left( -2\mu \right)
}{\Gamma \left( \frac{1}{2}-\mu -\kappa \right) }\mathrm{M}_{\kappa ,\mu
}\left( x\right) +\frac{\Gamma \left( 2\mu \right) }{\Gamma \left( \frac{1}{2%
}+\mu -\kappa \right) }\mathrm{M}_{\kappa ,-\mu }\left( x\right) ,
\label{W_k,mu_def}
\end{equation}%
where $\Gamma \left( x\right) $ denotes the \textit{gamma function}. These
functions, called Whittaker functions, are closely associated to the
following \textit{confluent hypergeometric function} (Kummer function):%
\begin{equation}
_{1}F_{1}\left( \left.
\begin{array}{c}
a \\
b%
\end{array}%
\right\vert x\right) =\frac{\Gamma \left( b\right) }{\Gamma \left( a\right) }%
\sum_{n=0}^{\infty }\frac{\Gamma \left( a+n\right) }{\Gamma \left(
b+n\right) }\frac{x^{n}}{n!},  \label{1F1_Whittaker_def}
\end{equation}%
where $_{p}F_{q}\left( \left.
\begin{array}{c}
a_{1},\ldots ,a_{p} \\
b_{1},\ldots ,b_{q}%
\end{array}%
\right\vert x\right) $ denotes the \textit{generalized hypergeometric
function}.

For particular values of the parameters $\kappa $ and $\mu $, the Whittaker
functions can be reduced to a variety of elementary and special functions.
Whittaker \cite{whittaker1903expression}\ discussed the connection of the
functions defined in (\ref{M_k,mu_def}) and (\ref{W_k,mu_def})\ with many
other special functions, such as the modified Bessel function, the
incomplete gamma functions, the parabolic cylinder function, the error
functions, the logarithmic and the cosine integrals, and the generalize
Hermite and Laguerre polynomials. Monographs and treatises dealing with
special functions \cite%
{erdelyi1953bateman,slater1960confluent,whittaker1920course,olver2010nist,magnus2013formulas,buchholz1969confluent,gradstein2015table,prudnikov1986integrals,oldham2009atlas}%
\ present the properties of the Whittaker functions with more or less
extension.

The Whittaker functions are frequently applied in various areas of
mathematical physics (see, for example \cite%
{chaudhry1992remarks,rusin2011green,omair2022family}), such as the
well-known solution of the Schr\"{o}dinger equation for the harmonic
oscillator \cite{lagarias2009schrodinger}.

$\mathrm{M}_{\kappa ,\mu }\left( x\right) $ and $\mathrm{W}_{\kappa ,\mu
}\left( x\right) $ are usually treated as functions of variable $x$ with
fixed values of the parameters $\kappa $ and $\mu $. However, there are few
investigations which consider $\kappa $ and $\mu $ as variables. For
instance, Laurenzi \cite{laurenzi1973derivatives}\ discussed methods to
calculate derivatives of $\mathrm{M}_{\kappa ,1/2}\left( x\right) $ and $%
\mathrm{W}_{\kappa ,1/2}\left( x\right) $ with respect to $\kappa $ when
this parameter is an integer. Using the Mellin transform, Buschman \cite%
{buschman1974finite}\ showed that the derivatives of certain Whittaker
functions with respect to the parameters can be expressed in finite sums of
Whittaker functions. L\'{o}pez and Sesma \cite{lopez1999whittaker}
considered the behaviour of $\mathrm{M}_{\kappa ,\mu }\left( x\right) $ as a
function of $\kappa $. They derived a convergent expansion in ascending
powers of $\kappa $, and an asymptotic expansion in descending powers of $%
\kappa $. Using series of Bessel functions and Buchholz polynomials, Abad
and Sesma \cite{abad2003successive} presented an algorithm for the
calculation of the $n$-th derivative of the Whittaker functions with respect
to parameter $\kappa $. Becker \cite{becker2009infinite}\ investigated
certain integrals with respect to parameter $\mu $. Ancarini and Gasaneo
\cite{ancarani2010derivatives} presented a general case of differentiation
of generalized hypergeometric functions with respect to the parameters in
terms of infinite series containing the digamma function. In addition,
Sofostasios and Brychkov \cite{sofotasios2018derivatives} considered
derivatives of hypergeometric functions and classical polynomials with
respect to the parameters.

In this paper, our main focus will be directed to the systematic
investigation of the first derivatives of $\mathrm{M}_{\kappa ,\mu }\left(
x\right) $ with respect to the parameters. We will mainly base our results
on two different approaches. The first one has to do with the series
representation of $\mathrm{M}_{\kappa ,\mu }\left( x\right) $, and the
second one has to do with the integral representations of $\mathrm{M}%
_{\kappa ,\mu }\left( x\right) $. Regarding the first approach, direct
differentiation of (\ref{M_k,mu_def}) with respect to the parameters leads
to infinite sums of quotients of digamma and gamma functions. It is possible
to calculate such sums in closed-form for particular values of the
parameters. The parameter differentiation of the integral representations of
$\mathrm{M}_{\kappa ,\mu }\left( x\right) $ leads to finite and infinite
integrals of elementary functions, such as products of algebraic functions,
exponential and logarithmic functions. These integrals are similar to those
investigated by K\"{o}lbig \cite{kolbig1987integral}\ and Geddes et al. \cite%
{geddes1990evaluation}. As in the case of the first approach, it is possible
to calculate such integrals in closed-form for some particular values of the
parameters.

In the Appendices, we calculate the first derivative of the incomplete gamma
functions $\gamma \left( \nu ,x\right) $ and $\Gamma \left( \nu ,x\right) $
with respect to the parameter $\nu $. These results will be used when we
calculate some of the integrals found in the second approach mentioned
before. Also, we calculate some new reduction formulas of the integral
Whittaker functions, which were recently introduced by us in \cite%
{apelblat2021integral}. They are defined in a similar way as other integral
functions in the mathematical literature:
\begin{eqnarray}
\mathrm{Mi}_{\kappa ,\mu }\left( x\right) &=&\int_{0}^{x}\frac{\mathrm{M}%
_{\kappa ,\mu }\left( t\right) }{t}dt,  \label{Mi_def} \\
\mathrm{mi}_{\kappa ,\mu }\left( x\right) &=&\int_{x}^{\infty }\frac{\mathrm{%
M}_{\kappa ,\mu }\left( t\right) }{t}dt.  \label{mi_def}
\end{eqnarray}

Finally, we also include a list of reduction formulas for the Whittaker
function $\mathrm{M}_{\kappa ,\mu }\left( x\right) $ in the Appendices.

\section{Parameter differentiation of $\mathrm{M}_{\protect\kappa ,\protect%
\mu }$ via Kummer function $_{1}F_{1}$}

As mentioned before, the Whittaker function $\mathrm{M}_{\kappa ,\mu }\left(
x\right) $ is closely related to the confluent hypergeometric function $%
_{1}F_{1}\left( a;b;x\right) $. Likewise, the parameter derivatives of $%
\mathrm{M}_{\kappa ,\mu }\left( x\right) $ are also related to the parameter
derivatives of $_{1}F_{1}\left( a;b;x\right) $. Let us introduce the
following notation set by Ancarini and Gasaneo \cite{ancarani2010derivatives}%
.

\begin{definition}
Define the parameter derivatives of the confluent hypergeometric function
as,
\begin{equation}
G^{\left( 1\right) }\left( \left.
\begin{array}{c}
a \\
b%
\end{array}%
\right\vert x\right) =\frac{\partial }{\partial a}\,\left[ _{1}F_{1}\left(
\left.
\begin{array}{c}
a \\
b%
\end{array}%
\right\vert x\right) \right] ,  \label{G(1)_def}
\end{equation}%
and%
\begin{equation}
H^{\left( 1\right) }\left( \left.
\begin{array}{c}
a \\
b%
\end{array}%
\right\vert x\right) =\frac{\partial }{\partial b}\,\left[ _{1}F_{1}\left(
\left.
\begin{array}{c}
a \\
b%
\end{array}%
\right\vert x\right) \right] .  \label{H(1)_def}
\end{equation}
\end{definition}

According to (\ref{1F1_Whittaker_def}), we have%
\begin{eqnarray*}
G^{\left( 1\right) }\left( \left.
\begin{array}{c}
a \\
b%
\end{array}%
\right\vert x\right) &=&\frac{\Gamma \left( b\right) }{\Gamma \left(
a\right) }\sum_{n=0}^{\infty }\frac{\Gamma \left( a+n\right) }{\Gamma \left(
b+n\right) }\left[ \psi \left( a+n\right) -\psi \left( a\right) \right]
\frac{x^{n}}{n!}, \\
H^{\left( 1\right) }\left( \left.
\begin{array}{c}
a \\
b%
\end{array}%
\right\vert x\right) &=&-\frac{\Gamma \left( b\right) }{\Gamma \left(
a\right) }\sum_{n=0}^{\infty }\frac{\Gamma \left( a+n\right) }{\Gamma \left(
b+n\right) }\left[ \psi \left( b+n\right) -\psi \left( b\right) \right]
\frac{x^{n}}{n!}.
\end{eqnarray*}

Since one of the integral representations of the confluent hypergeometric
function is \cite[Sect. 6.5.1]{magnus2013formulas}:%
\begin{eqnarray}
&&_{1}F_{1}\left( \left.
\begin{array}{c}
a \\
b%
\end{array}%
\right\vert x\right) =\frac{\Gamma \left( b\right) }{\Gamma \left( a\right)
\Gamma \left( b-a\right) }\int_{0}^{1}e^{xt}t^{a-1}\left( 1-t\right)
^{b-a-1}dt  \label{1F1_integral_representation} \\
&&\mathrm{Re}\,b>\mathrm{Re}\,a>0,  \notag
\end{eqnarray}%
by direct differentiation of (\ref{1F1_integral_representation}) with
respect to parameters $a$ and $b$, we obtain
\begin{eqnarray*}
G^{\left( 1\right) }\left( \left.
\begin{array}{c}
a \\
b%
\end{array}%
\right\vert x\right) &=&\left[ \psi \left( b\right) -\psi \left( a\right) %
\right] \,_{1}F_{1}\left( \left.
\begin{array}{c}
a \\
b%
\end{array}%
\right\vert x\right) \\
&&+\frac{\Gamma \left( b\right) }{\Gamma \left( a\right) \Gamma \left(
b-a\right) }\int_{0}^{1}e^{xt}t^{a-1}\left( 1-t\right) ^{b-a-1}\ln \left(
\frac{t}{1-t}\right) dt,
\end{eqnarray*}%
and%
\begin{eqnarray*}
H^{\left( 1\right) }\left( \left.
\begin{array}{c}
a \\
b%
\end{array}%
\right\vert x\right) &=&-\left[ \psi \left( b\right) -\psi \left( b-a\right) %
\right] \,_{1}F_{1}\left( \left.
\begin{array}{c}
a \\
b%
\end{array}%
\right\vert x\right) \\
&&+\frac{\Gamma \left( b\right) }{\Gamma \left( a\right) \Gamma \left(
b-a\right) }\int_{0}^{1}e^{xt}t^{a-1}\left( 1-t\right) ^{b-a-1}\ln \left(
1-t\right) dt.
\end{eqnarray*}

Since the main focus is the systematic investigation of the parameter
derivatives of $\mathrm{M}_{\kappa ,\mu }\left( x\right) $, we will present
these parameter derivatives as Theorems along the paper, and the
corresponding results for $G^{\left( 1\right) }\left( a;b;x\right) $ and $%
H^{\left( 1\right) }\left( a;b;x\right) $ as Corollaries. Also, note that
all the results regarding $G^{\left( 1\right) }\left( a;b;x\right) $ can be
transformed according to the next Theorem.

\begin{theorem}
The following transformation holds true:%
\begin{equation*}
G^{\left( 1\right) }\left( \left.
\begin{array}{c}
a \\
b%
\end{array}%
\right\vert x\right) =-e^{x}G^{\left( 1\right) }\left( \left.
\begin{array}{c}
b-a \\
b%
\end{array}%
\right\vert -x\right) .
\end{equation*}
\end{theorem}

\begin{proof}
Differentiate with respect to $a$ Kummer's transformation formula \cite[Eqn.
13.2.39]{olver2010nist}:%
\begin{equation*}
_{1}F_{1}\left( \left.
\begin{array}{c}
a \\
b%
\end{array}%
\right\vert x\right) =e^{x}\,_{1}F_{1}\left( \left.
\begin{array}{c}
b-a \\
b%
\end{array}%
\right\vert -x\right) ,
\end{equation*}%
to obtain the desired result.
\end{proof}

\subsection{Derivative with respect to the first parameter $\partial \mathrm{%
M}_{\protect\kappa ,\protect\mu }\left( x\right) /\partial \protect\kappa $}

Using (\ref{M_k,mu_def})\ and (\ref{1F1_Whittaker_def}), the first
derivative of $\mathrm{M}_{\kappa ,\mu }\left( x\right) $ with respect to
the first parameter $\kappa $ is%
\begin{eqnarray}
\frac{\partial \mathrm{M}_{\kappa ,\mu }\left( x\right) }{\partial \kappa }
&=&\psi \left( \frac{1}{2}+\mu -\kappa \right) \mathrm{M}_{\kappa ,\mu
}\left( x\right)  \label{DkM_general} \\
&&-\frac{\Gamma \left( 1+2\mu \right) }{\Gamma \left( \frac{1}{2}+\mu
-\kappa \right) }x^{\mu +1/2}e^{-x/2}S_{1}\left( \kappa ,\mu ,x\right) ,
\notag
\end{eqnarray}%
where $\psi \left( x\right) $ denotes the \textit{digamma function} and%
\begin{equation}
S_{1}\left( \kappa ,\mu ,x\right) =\sum_{n=0}^{\infty }\frac{\Gamma \left(
\frac{1}{2}+\mu -\kappa +n\right) }{\Gamma \left( 1+2\mu +n\right) }\psi
\left( \frac{1}{2}+\mu -\kappa +n\right) \frac{x^{n}}{n!}.  \label{S1_def}
\end{equation}

\begin{theorem}
For $\mu \neq -1/2$ and for $x\in
\mathbb{R}
$, the following parameter derivative formula of $\mathrm{M}_{\kappa ,\mu
}\left( x\right) $ holds true:%
\begin{equation}
\left. \frac{\partial \mathrm{M}_{\kappa ,\mu }\left( x\right) }{\partial
\kappa }\right\vert _{\kappa =-\mu -1/2}=-\frac{x^{\mu +3/2}}{2\mu +1}%
e^{x/2}\,_{2}F_{2}\left( \left.
\begin{array}{c}
1,1 \\
2\left( \mu +1\right) ,2%
\end{array}%
\right\vert -x\right) .  \label{DkM_k=mu-1/2}
\end{equation}
\end{theorem}

\begin{proof}
For $\kappa =-\mu -1/2$, Eqn. (\ref{DkM_general})\ becomes%
\begin{eqnarray*}
&&\left. \frac{\partial \mathrm{M}_{\kappa ,\mu }\left( x\right) }{\partial
\kappa }\right\vert _{\kappa =-\mu -1/2} \\
&=&x^{\mu +1/2}e^{-x/2}\left[ \psi \left( 1+2\mu \right) \sum_{n=0}^{\infty }%
\frac{x^{n}}{n!}-\sum_{n=0}^{\infty }\psi \left( 2\mu +1+n\right) \frac{x^{n}%
}{n!}\right] .
\end{eqnarray*}%
Apply \cite[Eqn. 6.2.1(60)]{brychkov2008handbook}
\begin{equation}
\sum_{k=0}^{\infty }\frac{t^{k}}{k!}\psi \left( k+a\right) =e^{t}\left[ \psi
\left( a\right) +\frac{t}{a}\,_{2}F_{2}\left( \left.
\begin{array}{c}
1,1 \\
a+1,2%
\end{array}%
\right\vert -t\right) \right] .  \label{Q_0(t,a)}
\end{equation}%
to obtain (\ref{DkM_k=mu-1/2}), as we wanted to prove.
\end{proof}

\begin{corollary}
For $a\in
\mathbb{R}
$, $a\neq 0$, and for $x\in
\mathbb{R}
$, the following reduction formula holds true:%
\begin{equation}
G^{\left( 1\right) }\left( \left.
\begin{array}{c}
a \\
a%
\end{array}%
\right\vert x\right) =\frac{x\,e^{x}}{a}\,_{2}F_{2}\left( \left.
\begin{array}{c}
1,1 \\
a+1,2%
\end{array}%
\right\vert -x\right) .  \label{G(1)_a_resultado}
\end{equation}
\end{corollary}

\begin{proof}
Direct differentiation of (\ref{M_k,mu_def})\ yields%
\begin{equation}
\frac{\partial \mathrm{M}_{\kappa ,\mu }\left( x\right) }{\partial \kappa }%
=-x^{\mu +1/2}e^{-x/2}G^{\left( 1\right) }\left( \left.
\begin{array}{c}
\frac{1}{2}+\mu -\kappa \\
1+2\mu%
\end{array}%
\right\vert x\right) ,  \label{DkM_G(1)}
\end{equation}%
thus comparing (\ref{DkM_G(1)})\ with $\kappa =-\mu -\frac{1}{2}$ to (\ref%
{DkM_k=mu-1/2}), we arrive at (\ref{G(1)_a_resultado}), as we wanted to
prove.
\end{proof}

Table \ref{Table_1}\ presents some explicit expressions for particular
values of (\ref{DkM_k=mu-1/2}), and for $x\in
\mathbb{R}
$, obtained with the help of MATHEMATICA\ program.

\begin{center}
\begin{table}[tbp] \centering%
\caption{Derivative of $M_{\kappa,\mu}$ with respect
to $\kappa$ by using (\ref{DkM_k=mu-1/2}).}%
\begin{tabular}{|c|c|c|}
\hline
$\kappa $ & $\mu $ & $\frac{\partial \mathrm{M}_{\kappa ,\mu }\left(
x\right) }{\partial \kappa }$ \\ \hline\hline
$-\frac{3}{4}$ & $\frac{1}{4}$ & $-\frac{2}{3}x^{7/4}e^{x/2}\,_{2}F_{2}%
\left( 1,1;\frac{5}{2},2;-x\right) $ \\ \hline
$-\frac{1}{2}$ & $0$ & $-\sqrt{x}e^{x/2}\left[ \gamma +\ln x+\mathrm{Shi}%
\left( x\right) -\mathrm{Chi}\left( x\right) \right] $ \\ \hline
$-\frac{1}{4}$ & $-\frac{1}{4}$ & $-2x^{5/4}e^{x/2}\,_{2}F_{2}\left( 1,1;%
\frac{3}{2},2;-x\right) $ \\ \hline
$-\frac{1}{6}$ & $-\frac{1}{3}$ & $-3x^{7/6}e^{x/2}\,_{2}F_{2}\left( 1,1;%
\frac{4}{3},2;-x\right) $ \\ \hline
$0$ & $\frac{1}{2}$ & $e^{-x/2}\left[ \mathrm{Shi}\left( x\right) +\mathrm{%
Chi}\left( x\right) -\ln x-\gamma \right] -e^{x/2}\left[ \mathrm{Shi}\left(
x\right) -\mathrm{Chi}\left( x\right) +\ln x+\gamma \right] $ \\ \hline
$\frac{1}{6}$ & $-\frac{2}{3}$ & $3x^{5/6}e^{x/2}\,_{2}F_{2}\left( 1,1;\frac{%
2}{3},2;-x\right) $ \\ \hline
$\frac{1}{2}$ & $1$ & $%
\begin{array}{l}
-\frac{2}{\sqrt{x}}\left\{ e^{x/2}\left[ \gamma +1+\ln x+\mathrm{Shi}\left(
x\right) -\mathrm{Chi}\left( x\right) \right] \right. \\
\quad +\left. e^{-x/2}\left( x+1\right) \left[ \gamma -1+\ln x-\mathrm{Shi}%
\left( x\right) -\mathrm{Chi}\left( x\right) \right] \right\}%
\end{array}%
$ \\ \hline
$1$ & $\frac{3}{2}$ & $%
\begin{array}{l}
-\frac{3}{x}\left\{ e^{-x/2}\left[ \left( x^{2}+2x+2\right) \left( \ln x-%
\mathrm{Shi}\left( x\right) -\mathrm{Chi}\left( x\right) +\gamma \right) %
\right] \right. \\
\quad +\left. e^{x/2}\left[ 2\ln x+2\,\mathrm{Shi}\left( x\right) -2\,%
\mathrm{Chi}\left( x\right) +x+2\gamma +3\right] \right\}%
\end{array}%
$ \\ \hline
\end{tabular}%
\label{Table_1}%
\end{table}%
\end{center}

Next, we present other reduction formula of $\partial \mathrm{M}_{\kappa
,\mu }\left( x\right) /\partial \kappa $ from the result found in \cite%
{laurenzi1973derivatives} for $x\in
\mathbb{R}
$,%
\begin{eqnarray}
&&\left. \frac{\partial \mathrm{M}_{\kappa ,\mu }\left( x\right) }{\partial
\kappa }\right\vert _{\kappa =n,\mu =1/2}  \label{Laurenzi_M} \\
&=&\left[ \ln \left\vert x\right\vert -\psi \left( n+1\right) -\mathrm{Ei}%
\left( x\right) \right] \,\mathrm{M}_{n,1/2}\left( x\right) +\sum_{\ell
=0}^{n-1}\left( a_{\ell }+b_{\ell }\,e^{x}\right) \,\mathrm{M}_{\ell
,1/2}\left( x\right) ,  \notag
\end{eqnarray}%
where $\mathrm{Ei}\left( x\right) $ denotes the exponential integral and for
$n,\ell =1,2,\ldots $%
\begin{equation}
a_{\ell }=\frac{1}{n}\left( \frac{n+\ell }{n-\ell }\right) ,  \label{a_l}
\end{equation}%
and%
\begin{equation}
b_{\ell }=\left\{
\begin{array}{ll}
\displaystyle%
\frac{1}{n}\sum_{k=0}^{n-\ell -1}\frac{\left( \ell \right) _{k}2^{k}}{\left(
\ell +n\right) _{k}}, & \ell =1,2,\ldots \\
0, & \ell =0.%
\end{array}%
\right.  \label{b_l}
\end{equation}

In order to calculate the finite sum given in (\ref{b_l}), we derive the
following Lemma.

\begin{lemma}
The following finite sum holds true $\forall n,\ell =1,2,\ldots $%
\begin{equation}
S\left( n,\ell \right) =\sum_{k=0}^{n-\ell -1}\frac{\left( \ell \right)
_{k}2^{k}}{\left( \ell +n\right) _{k}}=\mathrm{Re}\left[ _{2}F_{1}\left(
\left.
\begin{array}{c}
1,\ell \\
\ell +n%
\end{array}%
\right\vert 2\right) \right] .  \label{S_n,l}
\end{equation}
\end{lemma}

\begin{proof}
Split the sum in two as%
\begin{equation*}
S\left( n,\ell \right) =\underset{S_{1}\left( n,\ell \right) }{\underbrace{%
\sum_{k=0}^{\infty }\frac{\left( \ell \right) _{k}\left( 1\right) _{k}2^{k}}{%
k!\left( \ell +n\right) _{k}}}}-\underset{S_{2}\left( n,\ell \right) }{%
\underbrace{\sum_{k=n-\ell }^{\infty }\frac{\left( \ell \right) _{k}\left(
1\right) _{k}2^{k}}{k!\left( \ell +n\right) _{k}}}},
\end{equation*}%
where%
\begin{equation*}
S_{1}\left( n,\ell \right) =\,_{2}F_{1}\left( \left.
\begin{array}{c}
1,\ell \\
\ell +n%
\end{array}%
\right\vert 2\right) ,
\end{equation*}%
and%
\begin{eqnarray*}
S_{2}\left( n,\ell \right) &=&2^{n-\ell }\sum_{s=0}^{\infty }\frac{\left(
\ell \right) _{s+n-\ell }\left( 1\right) _{s}2^{s}}{s!\left( \ell +n\right)
_{s+n-\ell }} \\
&=&2^{n-\ell }\frac{\left( \ell \right) _{n}}{\left( n\right) _{n}}%
\sum_{s=0}^{\infty }\frac{\left( n\right) _{s}\left( 1\right) _{s}2^{s}}{%
s!\left( 2n\right) _{s}} \\
&=&2^{n-\ell }\frac{\left( \ell \right) _{n}}{\left( n\right) _{n}}%
\,_{2}F_{1}\left( \left.
\begin{array}{c}
1,n \\
2n%
\end{array}%
\right\vert 2\right) .
\end{eqnarray*}%
Take $a=1$, $b=n$, and $z=2$ in the quadratic transformation \cite[Eqn.
15.18.3]{olver2010nist}%
\begin{eqnarray*}
&&_{2}F_{1}\left( \left.
\begin{array}{c}
a,b \\
2b%
\end{array}%
\right\vert z\right) \\
&=&\left( 1-z\right) ^{-a/2}\,_{2}F_{1}\left( \left.
\begin{array}{c}
\frac{a}{2},b-\frac{a}{2} \\
b+\frac{1}{2}%
\end{array}%
\right\vert \frac{z^{2}}{4\left( z-1\right) }\right) ,
\end{eqnarray*}%
to obtain%
\begin{equation*}
_{2}F_{1}\left( \left.
\begin{array}{c}
1,n \\
2n%
\end{array}%
\right\vert 2\right) =i\,_{2}F_{1}\left( \left.
\begin{array}{c}
\frac{1}{2},n-\frac{1}{2} \\
n+\frac{1}{2}%
\end{array}%
\right\vert 1\right) .
\end{equation*}%
Now, apply Gauss's summation theorem \cite[Eqn. 15.4.20]{olver2010nist}\
\begin{eqnarray*}
_{2}F_{1}\left( \left.
\begin{array}{c}
a,b \\
c%
\end{array}%
\right\vert 1\right) &=&\frac{\Gamma \left( c\right) \Gamma \left(
c-a-b\right) }{\Gamma \left( c-a\right) \Gamma \left( c-b\right) }, \\
\mathrm{Re}\,\left( c-a-b\right) &>&0,
\end{eqnarray*}%
and the formula \cite[Eqn. 43:4:3]{oldham2009atlas}
\begin{equation*}
\Gamma \left( n+\frac{1}{2}\right) =\frac{\left( 2n-1\right) !!}{2^{n}}\sqrt{%
\pi },
\end{equation*}%
to arrive at%
\begin{equation*}
_{2}F_{1}\left( \left.
\begin{array}{c}
1,n \\
2n%
\end{array}%
\right\vert 2\right) =i\pi \,\frac{\left( 2n-1\right) !!}{2^{n}\left(
n-1\right) !}.
\end{equation*}%
Therefore, $S_{2}\left( n,\ell \right) $ is a pure imaginary number. Since $%
S\left( n,\ell \right) $ is a real number, we conclude that $S\left( n,\ell
\right) =\mathrm{Re}\left[ S_{1}\left( n,\ell \right) \right] $, as we
wanted to prove.
\end{proof}

\begin{theorem}
The following reduction formula holds true for $n=1,2,\ldots $ and $x\in
\mathbb{R}
$,
\begin{eqnarray}
&&\left. \frac{\partial \mathrm{M}_{\kappa ,\mu }\left( x\right) }{\partial
\kappa }\right\vert _{\kappa =n,\mu =1/2}  \label{DkMnmedio} \\
&=&\frac{2}{n}\sinh \left( \frac{x}{2}\right) +\frac{x\,e^{-x/2}}{n}  \notag
\\
&&\left\{ \left[ \ln \left\vert x\right\vert +\gamma -H_{n}-\mathrm{Ei}%
\left( x\right) \right] \,L_{n-1}^{\left( 1\right) }\left( x\right) \right.
\notag \\
&&+\left. \sum_{\ell =1}^{n-1}\left( \frac{n+\ell }{n-\ell }-e^{x}\,\mathrm{%
Re}\left[ _{2}F_{1}\left( \left.
\begin{array}{c}
1,\ell \\
\ell +n%
\end{array}%
\right\vert 2\right) \right] \right) \frac{L_{\ell -1}^{\left( 1\right)
}\left( x\right) }{\ell }\right\} ,  \notag
\end{eqnarray}%
where $L_{n}^{\left( \alpha \right) }\left( x\right) $ denotes the Laguerre
polynomials (\ref{Laguerre_def}) and $H_{n}=\sum_{k=1}^{n}\frac{1}{k}$ the $%
n $-th harmonic number.
\end{theorem}

\begin{proof}
From (\ref{S_n,l}) and (\ref{b_l}), we see that%
\begin{equation}
b_{\ell }=\mathrm{Re}\left[ _{2}F_{1}\left( \left.
\begin{array}{c}
1,\ell \\
\ell +n%
\end{array}%
\right\vert 2\right) \right] ,\quad \ell =1,2,\ldots  \label{b_l_Re}
\end{equation}%
Also, according to \cite[Eqn. 13.18.1]{olver2010nist}%
\begin{equation}
\mathrm{M}_{0,1/2}\left( x\right) =2\sinh \left( \frac{x}{2}\right) .
\label{M_0,1/2}
\end{equation}%
In addition, performing the transformations $\kappa \rightarrow \kappa +1$, $%
\kappa \rightarrow 0,$ and $n\rightarrow n-1$ in (\ref{M_k+n,k-1/2}), we
obtain $\forall n=1,2,\ldots $
\begin{equation}
\mathrm{M}_{n,1/2}\left( x\right) =\frac{x\,e^{-x/2}}{n}L_{n-1}^{\left(
1\right) }\left( x\right) .  \label{M_n_1/2}
\end{equation}%
Finally, we have\ for $n=1,2,\ldots $ \cite[Eqn. 1.3.7]{lebedev1965special}%
\begin{equation}
\psi \left( n+1\right) =-\gamma +H_{n}.  \label{Psi(n+1)}
\end{equation}%
Insert (\ref{a_l}) and (\ref{b_l})-(\ref{Psi(n+1)})\ in (\ref{Laurenzi_M})\
to arrive at (\ref{DkMnmedio}), as we wanted to prove.
\end{proof}

\begin{corollary}
The following reduction formula holds true for $n=1,2,\ldots $ and $x\in
\mathbb{R}
$,
\begin{eqnarray*}
&&G^{\left( 1\right) }\left( \left.
\begin{array}{c}
1-n \\
2%
\end{array}%
\right\vert x\right) \\
&=&\frac{1}{n}\left\{ \frac{1-e^{x}}{x}-\left[ \ln \left\vert x\right\vert
+\gamma -H_{n}-\mathrm{Ei}\left( x\right) \right] \,L_{n-1}^{\left( 1\right)
}\left( x\right) \right. \\
&&-\left. \sum_{\ell =1}^{n-1}\left( \frac{n+\ell }{n-\ell }-e^{x}\,\mathrm{%
Re}\left[ _{2}F_{1}\left( \left.
\begin{array}{c}
1,\ell \\
\ell +n%
\end{array}%
\right\vert 2\right) \right] \right) \frac{L_{\ell -1}^{\left( 1\right)
}\left( x\right) }{\ell }\right\} .
\end{eqnarray*}
\end{corollary}

\begin{proof}
Consider (\ref{DkM_G(1)}) and (\ref{DkMnmedio})\ to arrive at the desired
result.
\end{proof}

In Table \ref{TableDkMmedio}\ we collect some particular cases of (\ref%
{DkMnmedio}) for $x\in
\mathbb{R}
$ obtained with the help of MATHEMATICA\ program.

\begin{center}
\begin{table}[htbp] \centering%
\caption{Derivative of $M_{\kappa,\mu}$ with respect
to $\kappa$ by using (\ref{DkMnmedio}).}%
\begin{tabular}{|c|c|c|}
\hline
$\kappa $ & $\mu $ & $\frac{\partial \mathrm{M}_{\kappa ,\mu }\left(
x\right) }{\partial \kappa }$ \\ \hline\hline
$1$ & $\frac{1}{2}$ & $x\,e^{-x/2}\left[ \ln \left\vert x\right\vert -%
\mathrm{Ei}(x)+\gamma -1\right] +2\sinh \left( \frac{x}{2}\right) $ \\ \hline
$2$ & $\frac{1}{2}$ & $\frac{1}{2}x\,e^{-x/2}\left\{ \left( 2-x\right) \left[
\ln \left\vert x\right\vert -\,\mathrm{Ei}(x)+\gamma -\frac{3}{2}\right]
-e^{x}+3\right\} +\sinh \left( \frac{x}{2}\right) $ \\ \hline
$3$ & $\frac{1}{2}$ & $%
\begin{array}{c}
\frac{1}{6}x\,e^{-x/2}\left[ \left( x^{2}-6x+6\right) \left( \ln \left\vert
x\right\vert -\mathrm{Ei}(x)+\gamma -\frac{11}{6}\right) \right. \\
\quad +\left. (e^{x}-5)(x-2)-3e^{x}+4\right] +\frac{2}{3}\sinh \left( \frac{x%
}{2}\right)%
\end{array}%
$ \\ \hline
\end{tabular}%
\label{TableDkMmedio}%
\end{table}%
\end{center}

\subsection{Derivative with respect to the second parameter $\partial
\mathrm{M}_{\protect\kappa ,\protect\mu }\left( x\right) /\partial \protect%
\mu $}

Using (\ref{M_k,mu_def})\ and (\ref{1F1_Whittaker_def}), the first
derivative of $\mathrm{M}_{\kappa ,\mu }\left( x\right) $ with respect to
the parameter $\mu $ is%
\begin{eqnarray}
&&\frac{\partial \mathrm{M}_{\kappa ,\mu }\left( x\right) }{\partial \mu }
\label{DmuM_1} \\
&=&\left[ \ln x+2\,\psi \left( 1+2\mu \right) -\psi \left( \frac{1}{2}+\mu
-\kappa \right) \right] \mathrm{M}_{\kappa ,\mu }\left( x\right)  \notag \\
&&+x^{\mu +1/2}e^{-x/2}\frac{\Gamma \left( 1+2\mu \right) }{\Gamma \left(
\frac{1}{2}+\mu -\kappa \right) }\left[ S_{1}\left( \kappa ,\mu ,x\right)
-S_{2}\left( \kappa ,\mu ,x\right) \right] ,  \notag
\end{eqnarray}%
where $S_{1}\left( \kappa ,\mu ,x\right) $ is given in (\ref{S1_def})\ and
the series $S_{2}\left( \kappa ,\mu ,x\right) $ is%
\begin{equation}
S_{2}\left( \kappa ,\mu ,x\right) =2\sum_{n=0}^{\infty }\frac{\Gamma \left(
\frac{1}{2}+\mu -\kappa +n\right) }{\Gamma \left( 1+2\mu +n\right) }\psi
\left( 1+2\mu +n\right) \frac{x^{n}}{n!}.  \label{S2_def}
\end{equation}

\begin{theorem}
For $\mu \neq -1/2$ and $x\in
\mathbb{R}
$, the following parameter derivative formula of $\mathrm{M}_{\kappa ,\mu
}\left( x\right) $ holds true:%
\begin{eqnarray}
&&\left. \frac{\partial \mathrm{M}_{\kappa ,\mu }\left( x\right) }{\partial
\mu }\right\vert _{\kappa =-\mu -1/2}  \label{DmuM_k=-mu-1/2} \\
&=&x^{\mu +1/2}e^{x/2}\left[ \ln x-\frac{x}{1+2\mu }\,_{2}F_{2}\left( \left.
\begin{array}{c}
1,1 \\
2\left( \mu +1\right) ,2%
\end{array}%
\right\vert -x\right) \right] .  \notag
\end{eqnarray}
\end{theorem}

\begin{proof}
For $\kappa =-\mu -1/2$, we have $S_{2}\left( \kappa ,\mu ,x\right)
=2\,S_{1}\left( \kappa ,\mu ,x\right) $ and therefore (\ref{DmuM_1})\ becomes%
\begin{eqnarray*}
&&\left. \frac{\partial \mathrm{M}_{\kappa ,\mu }\left( x\right) }{\partial
\mu }\right\vert _{\kappa =-\mu -1/2} \\
&=&\left[ \ln x+\psi \left( 1+2\mu \right) \right] \mathrm{M}_{-\mu -1/2,\mu
}\left( x\right) -x^{\mu +1/2}e^{-x/2}S_{1}\left( -\mu -\frac{1}{2},\mu
,x\right) ,
\end{eqnarray*}%
where
\begin{equation*}
S_{1}\left( -\mu -\frac{1}{2},\mu ,x\right) =\sum_{n=0}^{\infty }\psi \left(
1+2\mu +n\right) \frac{x^{n}}{n!},
\end{equation*}%
thus, using (\ref{Q_0(t,a)}),%
\begin{eqnarray}
&&\left. \frac{\partial \mathrm{M}_{\kappa ,\mu }\left( x\right) }{\partial
\mu }\right\vert _{\kappa =-\mu -1/2}  \label{DmuM_2} \\
&=&\left[ \ln x+\psi \left( 1+2\mu \right) \right] \mathrm{M}_{-\mu -1/2,\mu
}\left( x\right)  \notag \\
&&-x^{\mu +1/2}e^{x/2}\left[ \psi \left( 1+2\mu \right) +\frac{x}{1+2\mu }%
\,_{2}F_{2}\left( \left.
\begin{array}{c}
1,1 \\
2\mu +2,2%
\end{array}%
\right\vert -x\right) \right] .  \notag
\end{eqnarray}%
Since, according to (\ref{M_k,mu_def}) and (\ref{1F1_Whittaker_def}),
\begin{equation*}
\mathrm{M}_{-\mu -1/2,\mu }\left( x\right) =x^{\mu +1/2}e^{x/2},
\end{equation*}%
then (\ref{DmuM_2})\ takes the simple form given in (\ref{DmuM_k=-mu-1/2}),
as we wanted to prove.
\end{proof}

\begin{corollary}
For $a\in
\mathbb{R}
$, $a\neq 0$, and $x\in
\mathbb{R}
$, the following reduction formula holds true:%
\begin{equation}
H^{\left( 1\right) }\left( \left.
\begin{array}{c}
a \\
a%
\end{array}%
\right\vert x\right) =-\frac{x\,e^{x}}{a}\,_{2}F_{2}\left( \left.
\begin{array}{c}
1,1 \\
a+1,2%
\end{array}%
\right\vert -x\right) .  \label{H(1)_a_resultado}
\end{equation}
\end{corollary}

\begin{proof}
Direct differentiation of (\ref{M_k,mu_def})\ yields%
\begin{eqnarray}
\frac{\partial \mathrm{M}_{\kappa ,\mu }\left( x\right) }{\partial \mu }
&=&\ln x\,\mathrm{M}_{\kappa ,\mu }\left( x\right) +\,x^{\mu +1/2}e^{-x/2}
\label{DmM_H(1)} \\
&&\left[ G^{\left( 1\right) }\left( \left.
\begin{array}{c}
\frac{1}{2}+\mu -\kappa \\
1+2\mu%
\end{array}%
\right\vert x\right) +2\,H^{\left( 1\right) }\left( \left.
\begin{array}{c}
\frac{1}{2}+\mu -\kappa \\
1+2\mu%
\end{array}%
\right\vert x\right) \right] ,  \notag
\end{eqnarray}%
thus comparing (\ref{DmM_H(1)})\ with $\kappa =-\mu -\frac{1}{2}$ to (\ref%
{DmuM_k=-mu-1/2})\ and taking into account (\ref{G(1)_a_resultado}), we
arrive at (\ref{H(1)_a_resultado}), as we wanted to prove.
\end{proof}

Using (\ref{DmuM_k=-mu-1/2}), the derivative of $\mathrm{M}_{\kappa ,\mu
}\left( x\right) $ with respect $\mu $ has been calculated for particular
values of $\kappa $ and $\mu $, with $x\in
\mathbb{R}
$, using the help of MATHEMATICA, and they are presented in Table \ref%
{Table_2A}.

\begin{center}
\begin{table}[tbp] \centering%
\caption{Derivative of $M_{\kappa,\mu}$ with respect
to $\mu$ by using (\ref{DmuM_k=-mu-1/2}).}%
\begin{tabular}{|c|c|c|}
\hline
$\kappa $ & $\mu $ & $\frac{\partial \mathrm{M}_{\kappa ,\mu }\left(
x\right) }{\partial \mu }$ \\ \hline\hline
$-\frac{3}{2}$ & $1$ & $\frac{1}{\sqrt{x}}\left\{ e^{x/2}\left[ x^{2}\left(
\mathrm{Chi}\left( x\right) -\mathrm{Shi}\left( x\right) -\gamma \right) +%
\frac{3}{2}x^{2}-2x+1\right] +e^{-x/2}\left( x-1\right) \right\} $ \\ \hline
$-1$ & $\frac{1}{2}$ & $x\,e^{x/2}\left[ \mathrm{Chi}\left( x\right) -%
\mathrm{Shi}\left( x\right) -\gamma +1\right] -2\sinh \left( \frac{x}{2}%
\right) $ \\ \hline
$-\frac{3}{4}$ & \multicolumn{1}{|c|}{$\frac{1}{4}$} & $e^{x/2}x^{3/4}\left[
\ln x-\frac{2}{3}x\,_{2}F_{2}\left( \left.
\begin{array}{c}
1,1 \\
2,\frac{5}{2}%
\end{array}%
\right\vert -x\right) \right] $ \\ \hline
$-\frac{1}{2}$ & \multicolumn{1}{|c|}{$0$} & $e^{x/2}\sqrt{x}\left[ \mathrm{%
Chi}\left( x\right) -\mathrm{Shi}\left( x\right) -\gamma \right] $ \\ \hline
$-\frac{1}{4}$ & \multicolumn{1}{|c|}{$-\frac{1}{4}$} & $e^{x/2}x^{1/4}\left[
\ln x-2x\,_{2}F_{2}\left( \left.
\begin{array}{c}
1,1 \\
2,\frac{3}{2}%
\end{array}%
\right\vert -x\right) \right] $ \\ \hline
$-\frac{1}{6}$ & \multicolumn{1}{|c|}{$-\frac{1}{3}$} & $e^{x/2}x^{1/6}\left[
\ln x-3x\,_{2}F_{2}\left( \left.
\begin{array}{c}
1,1 \\
2,\frac{4}{3}%
\end{array}%
\right\vert -x\right) \right] $ \\ \hline
$\frac{1}{6}$ & \multicolumn{1}{|c|}{$-\frac{2}{3}$} & $e^{x/2}x^{-1/6}\left[
\ln x+3x\,_{2}F_{2}\left( \left.
\begin{array}{c}
1,1 \\
2,\frac{2}{3}%
\end{array}%
\right\vert -x\right) \right] $ \\ \hline
\end{tabular}%
\label{Table_2A}%
\end{table}%
\end{center}

Note that for $\mu =-1/2$, we obtain an indeterminate expression in (\ref%
{DmuM_k=-mu-1/2}). For this case, we present the following result.

\begin{theorem}
The following parameter derivative formula of $\mathrm{M}_{\kappa ,\mu
}\left( x\right) $ holds true for $x\in
\mathbb{R}
$:%
\begin{eqnarray}
&&\left. \frac{\partial \mathrm{M}_{\kappa ,\mu }\left( x\right) }{\partial
\mu }\right\vert _{\kappa =0}  \label{DmM0} \\
&=&4^{\mu }\sqrt{x}\,\Gamma \left( 1+\mu \right) \left\{ I_{\mu }\left(
\frac{x}{2}\right) \left[ \ln 4+\psi \left( 1+\mu \right) \right] +\frac{%
\partial I_{\mu }\left( x/2\right) }{\partial \mu }\right\} ,  \notag
\end{eqnarray}%
where $I_{\nu }\left( x\right) $ denotes the modified Bessel function.
\end{theorem}

\begin{proof}
Differentiate with respect to $\mu $ the expression \cite[Eqn. 13.18.8]%
{olver2010nist}
\begin{equation}
\mathrm{M}_{0,\mu }\left( x\right) =4^{\mu }\,\Gamma \left( 1+\mu \right)
\sqrt{x}I_{\mu }\left( \frac{x}{2}\right) ,  \label{M_0_mu}
\end{equation}%
to obtain (\ref{DmM0}), as we wanted to prove.
\end{proof}

The order derivative of the modified Bessel function $I_{\mu }\left(
x\right) $ is given in terms of the Meijer-G function and the generalized
hypergeometric function $\forall \,\mathrm{Re\,}x>0,\mu \geq 0$ \cite%
{gonzalez2018closed}:%
\begin{eqnarray}
\frac{\partial I_{\mu }\left( x\right) }{\partial \mu } &=&-\frac{\mu
\,I_{\mu }\left( x\right) }{2\sqrt{\pi }}G_{2,4}^{3,1}\left( x^{2}\left\vert
\begin{array}{c}
\frac{1}{2},1 \\
0,0,\mu ,-\mu%
\end{array}%
\right. \right)  \label{DmuI_Meijer} \\
&&-\frac{K_{\mu }\left( x\right) }{\Gamma ^{2}\left( \mu +1\right) }\left(
\frac{x}{2}\right) ^{2\mu }\,_{2}F_{3}\left( \left.
\begin{array}{c}
\mu ,\mu +\frac{1}{2} \\
\mu +1,\mu +1,2\mu +1%
\end{array}%
\right\vert x^{2}\right) ,  \notag
\end{eqnarray}%
where $K_{\nu }\left( x\right) $ is the \textit{modified Bessel function of
the second kind};\ or in terms of generalized hypergeometric functions only $%
\forall \,\mathrm{Re\,}x>0,\mu >0$, $\mu \notin
\mathbb{Z}
$ \cite{brychkov2016higher}:%
\begin{eqnarray}
&&\frac{\partial I_{\mu }\left( x\right) }{\partial \mu }  \label{DmuI_Hyper}
\\
&=&I_{\mu }\left( x\right) \left[ \frac{x^{2}}{4\left( 1-\mu ^{2}\right) }%
\,_{3}F_{4}\left( \left.
\begin{array}{c}
1,1,\frac{3}{2} \\
2,2,2-\mu ,2+\mu%
\end{array}%
\right\vert x^{2}\right) +\ln \left( \frac{x}{2}\right) -\psi \left( \mu
\right) -\frac{1}{2\mu }\right]  \notag \\
&&-I_{-\mu }\left( x\right) \frac{\pi \csc \left( \pi \mu \right) }{%
2\,\Gamma ^{2}\left( \mu +1\right) }\left( \frac{x}{2}\right) ^{2\mu
}\,_{2}F_{3}\left( \left.
\begin{array}{c}
\mu ,\mu +\frac{1}{2} \\
\mu +1,\mu +1,2\mu +1%
\end{array}%
\right\vert x^{2}\right) .  \notag
\end{eqnarray}

There are different expressions for the order derivatives of the Bessel
functions \cite{apelblat1985integral,brychkov2005derivatives}. This subject
is summarized in \cite{apelblat2020bessel}, where general results are
presented in terms of convolution integrals, and order derivatives of Bessel
functions are found for particular values of the order.

Using (\ref{DmM0}), (\ref{DmuI_Meijer})\ and (\ref{DmuI_Hyper}), some
derivatives of $\mathrm{M}_{\kappa ,\mu }\left( x\right) $ with respect $\mu
$ has been calculated for $x\in
\mathbb{R}
$ with the help of MATHEMATICA, and they are presented in Table \ref{Table_2}%
.

\begin{center}
\begin{table}[tbp] \centering%
\caption{Derivative of $M_{\kappa,\mu}$ with respect
to $\mu$ by using (\ref{DmM0}).}%
\begin{tabular}{|c|c|c|}
\hline
$\kappa $ & $\mu $ & $\frac{\partial \mathrm{M}_{\kappa ,\mu }\left(
x\right) }{\partial \mu }$ \\ \hline\hline
$0$ & $-\frac{1}{2}$ & $\left[ \mathrm{Chi}\left( x\right) -\gamma \right]
\cosh \left( \frac{x}{2}\right) -\frac{2}{x}\sinh ^{3}\left( \frac{x}{2}%
\right) $ \\ \hline
$0$ & $0$ & $\sqrt{x}\left[ \left( \ln 4-\gamma \right) I_{0}\left( \frac{x}{%
2}\right) -K_{0}\left( \frac{x}{2}\right) \right] $ \\ \hline
$0$ & $\frac{1}{4}$ & $%
\begin{array}{l}
\frac{x^{3/4}}{15}\,_{0}F_{1}\left( ;\frac{5}{4};\frac{x^{2}}{16}\right) %
\left[ x^{2}\,_{3}F_{4}\left( 1,1,\frac{3}{2};\frac{7}{4},2,2,\frac{9}{4};%
\frac{x^{2}}{4}\right) +15\left( \ln x+2\right) \right] \\
-\frac{2\pi \,x}{\Gamma \left( \frac{1}{4}\right) }I_{-\frac{1}{4}}\left(
\frac{x}{2}\right) \,_{2}F_{3}\left( \frac{1}{4},\frac{3}{4};\frac{5}{4},%
\frac{5}{4},\frac{3}{2};\frac{x^{2}}{4}\right)%
\end{array}%
$ \\ \hline
$0$ & $\frac{1}{3}$ & $%
\begin{array}{l}
\frac{x^{5/6}}{128}\left\{ \,_{0}F_{1}\left( ;\frac{4}{3};\frac{x^{2}}{16}%
\right) \left[ 9x^{2}\,_{3}F_{4}\left( 1,1,\frac{3}{2};\frac{5}{3},2,2,\frac{%
7}{3};\frac{x^{2}}{4}\right) +64\left( 2\ln x+3\right) \right] \right. \\
-\left. 192\,_{0}F_{1}\left( ;\frac{2}{3};\frac{x^{2}}{16}\right)
\,_{2}F_{3}\left( \frac{1}{3},\frac{5}{6};\frac{4}{3},\frac{4}{3},\frac{5}{3}%
;\frac{x^{2}}{4}\right) \right\}%
\end{array}%
$ \\ \hline
$0$ & $\frac{1}{2}$ & $2\left[ \mathrm{Chi}\left( x\right) -\gamma +2\right]
\sinh \left( \frac{x}{2}\right) -2\,\mathrm{Shi}\left( x\right) \cosh \left(
\frac{x}{2}\right) $ \\ \hline
$0$ & $\frac{2}{3}$ & $%
\begin{array}{l}
\frac{x^{7/6}}{80}\left\{ \,_{0}F_{1}\left( ;\frac{5}{3};\frac{x^{2}}{16}%
\right) \left[ 9x^{2}\,_{3}F_{4}\left( 1,1,\frac{3}{2};\frac{4}{3},2,2,\frac{%
8}{3};\frac{x^{2}}{4}\right) +80\ln x+60\right] \right. \\
-\left. 60\,_{0}F_{1}\left( ;\frac{1}{3};\frac{x^{2}}{16}\right)
\,_{2}F_{3}\left( \frac{2}{3},\frac{7}{6};\frac{5}{3},\frac{5}{3},\frac{7}{3}%
;\frac{x^{2}}{4}\right) \right\}%
\end{array}%
$ \\ \hline
$0$ & $\frac{3}{4}$ & $%
\begin{array}{l}
\frac{x^{5/4}}{21}\,_{0}F_{1}\left( ;\frac{7}{4};\frac{x^{2}}{16}\right) %
\left[ 3x^{2}\,_{3}F_{4}\left( 1,1,\frac{3}{2};\frac{5}{4},2,2,\frac{11}{4};%
\frac{x^{2}}{4}\right) +21\ln x+14\right] \\
-\frac{\pi \,x^{2}}{4\,\Gamma \left( \frac{7}{4}\right) }I_{-\frac{3}{4}%
}\left( \frac{x}{2}\right) \,_{2}F_{3}\left( \frac{3}{4},\frac{5}{4};\frac{7%
}{4},\frac{7}{4},\frac{5}{2};\frac{x^{2}}{4}\right)%
\end{array}%
$ \\ \hline
$0$ & $1$ & $%
\begin{array}{l}
4\sqrt{x}\left\{ I_{1}\left( \frac{x}{2}\right) \left[ 1-\gamma +\ln 4-\frac{%
1}{2\sqrt{\pi }}G_{1,3}^{2,1}\left( \frac{x^{2}}{4};\frac{1}{2}%
;0,0,-1\right) \right] \right. \\
\quad -\left. K_{1}\left( \frac{x}{2}\right) \left[ I_{0}^{2}\left( \frac{x}{%
2}\right) -I_{1}^{2}\left( \frac{x}{2}\right) -1\right] \right\}%
\end{array}%
$ \\ \hline
$0$ & $\frac{3}{2}$ & $%
\begin{array}{l}
\frac{4}{x}\left\{ \sinh \left( \frac{x}{2}\right) \left[ 6\gamma -6\,%
\mathrm{Chi}\left( x\right) -3x\,\mathrm{Shi}\left( x\right) -28\right]
\right. \\
\quad +\left. \cosh \left( \frac{x}{2}\right) \left[ \left( 3\,\mathrm{Chi}%
\left( x\right) +8-3\gamma \right) x+6\,\mathrm{Shi}\left( x\right) \right]
\right\}%
\end{array}%
$ \\ \hline
$0$ & $2$ & $%
\begin{array}{l}
32\sqrt{x}\left\{ I_{2}\left( \frac{x}{2}\right) \left[ \frac{3}{2}-\gamma
+\ln 4-\frac{1}{\sqrt{\pi }}G_{2,4}^{3,1}\left( \frac{x^{2}}{4};\frac{1}{2}%
,1;0,0,2,-2\right) \right] \right. \\
\quad +\left. K_{2}\left( \frac{x}{2}\right) \left[ 2\,_{1}F_{2}\left( \frac{%
1}{2};1,3;\frac{x^{2}}{4}\right) -\,_{2}F_{3}\left( \frac{1}{2},2;1,1,3;%
\frac{x^{2}}{4}\right) -1\right] \right\}%
\end{array}%
$ \\ \hline
\end{tabular}%
\label{Table_2}%
\end{table}%
\end{center}

\section{Parameter differentiation of $\mathrm{M}_{\protect\kappa ,\protect%
\mu }$ via integral representations}

\subsection{Derivative with respect to the first parameter $\partial \mathrm{%
M}_{\protect\kappa ,\protect\mu }\left( x\right) /\partial \protect\kappa $}

Integral representations of $\mathrm{M}_{\kappa ,\mu }\left( x\right) $ can
be obtained via integral representations of confluent hypergeometric
function \cite[Sect. 7.4.1]{magnus2013formulas}, thus
\begin{eqnarray}
&&\mathrm{M}_{\kappa ,\mu }\left( x\right)  \notag \\
&=&\frac{\,x^{\mu +1/2}e^{-x/2}}{\mathrm{B}\left( \mu +\kappa +\frac{1}{2}%
,\mu -\kappa +\frac{1}{2}\right) }\int_{0}^{1}e^{xt}t^{\mu -\kappa
-1/2}\left( 1-t\right) ^{\mu +\kappa -1/2}dt  \label{M_k,mu_int_1} \\
&=&\frac{\,x^{\mu +1/2}e^{x/2}}{\mathrm{B}\left( \mu +\kappa +\frac{1}{2}%
,\mu -\kappa +\frac{1}{2}\right) }\int_{0}^{1}e^{-xt}t^{\mu +\kappa
-1/2}\left( 1-t\right) ^{\mu -\kappa -1/2}dt  \label{M_k,mu_int_2} \\
&&\mathrm{Re}\left( \mu \pm \kappa +\frac{1}{2}\right) >0,  \notag
\end{eqnarray}%
where
\begin{equation}
\mathrm{B}\left( a,b\right) =\frac{\Gamma \left( a\right) \Gamma \left(
b\right) }{\Gamma \left( a+b\right) }  \label{Beta_def}
\end{equation}%
denotes the beta function. In order to calculate the first derivative of $%
\mathrm{M}_{\kappa ,\mu }\left( x\right) $ with respect to parameter $\kappa
$, let us introduce the following finite logarithmic integrals.

\begin{definition}
\begin{eqnarray}
I_{1}\left( \kappa ,\mu ;x\right) &=&\int_{0}^{1}e^{xt}t^{\mu -\kappa
-1/2}\left( 1-t\right) ^{\mu +\kappa -1/2}\ln \left( \frac{1-t}{t}\right) dt,
\label{I1_def} \\
I_{2}\left( \kappa ,\mu ;x\right) &=&\int_{0}^{1}e^{-xt}t^{\mu +\kappa
-1/2}\left( 1-t\right) ^{\mu -\kappa -1/2}\ln \left( \frac{t}{1-t}\right) dt.
\label{I2_def}
\end{eqnarray}
\end{definition}

Differentiation of (\ref{M_k,mu_int_1}) and (\ref{M_k,mu_int_2})\ with
respect to parameter $\kappa $ yields respectively%
\begin{eqnarray}
\frac{\partial \mathrm{M}_{\kappa ,\mu }\left( x\right) }{\partial \kappa }
&=&\left[ \psi \left( \mu -\kappa +\frac{1}{2}\right) -\psi \left( \mu
+\kappa +\frac{1}{2}\right) \right] \mathrm{M}_{\kappa ,\mu }\left( x\right)
\label{DkM_I1} \\
&&+\frac{\,x^{\mu +1/2}e^{-x/2}}{\mathrm{B}\left( \mu +\kappa +\frac{1}{2}%
,\mu -\kappa +\frac{1}{2}\right) }I_{1}\left( \kappa ,\mu ;x\right)  \notag
\\
&=&\left[ \psi \left( \mu -\kappa +\frac{1}{2}\right) -\psi \left( \mu
+\kappa +\frac{1}{2}\right) \right] \mathrm{M}_{\kappa ,\mu }\left( x\right)
\label{DkM_I2} \\
&&+\frac{\,x^{\mu +1/2}e^{x/2}}{\mathrm{B}\left( \mu +\kappa +\frac{1}{2}%
,\mu -\kappa +\frac{1}{2}\right) }I_{2}\left( \kappa ,\mu ;x\right) ,  \notag
\end{eqnarray}

Note that, from (\ref{DkM_I1}) and (\ref{DkM_I2}), we have%
\begin{equation}
I_{2}\left( \kappa ,\mu ;x\right) =e^{-x}I_{1}\left( \kappa ,\mu ;x\right) .
\label{I1->I2}
\end{equation}

Likewise, we can depart from other integral respresentations of $\mathrm{M}%
_{\kappa ,\mu }\left( x\right) $ \cite[Sect. 7.4.1]{magnus2013formulas}%
\footnote{%
There are some typos in this reference regarding these integral
representations.}, to obtain%
\begin{eqnarray}
\mathrm{M}_{\kappa ,\mu }\left( x\right) &=&\frac{2^{-2\mu }\,x^{\mu +1/2}}{%
\mathrm{B}\left( \mu +\kappa +\frac{1}{2},\mu -\kappa +\frac{1}{2}\right) }
\label{M_k,mu_int_3} \\
&&\int_{-1}^{1}e^{xt/2}\left( 1+t\right) ^{\mu -\kappa -1/2}\left(
1-t\right) ^{\mu +\kappa -1/2}dt  \notag \\
&=&\frac{2^{-2\mu }\,x^{\mu +1/2}}{\mathrm{B}\left( \mu +\kappa +\frac{1}{2}%
,\mu -\kappa +\frac{1}{2}\right) }  \label{M_k,mu_int_4} \\
&&\int_{-1}^{1}e^{-xt/2}\left( 1+t\right) ^{\mu +\kappa -1/2}\left(
1-t\right) ^{\mu -\kappa -1/2}dt  \notag \\
&&\mathrm{Re}\left( \mu \pm \kappa +\frac{1}{2}\right) >0,  \notag
\end{eqnarray}%
and consequently, we have%
\begin{eqnarray}
\frac{\partial \mathrm{M}_{\kappa ,\mu }\left( x\right) }{\partial \kappa }
&=&\left[ \psi \left( \mu -\kappa +\frac{1}{2}\right) -\psi \left( \mu
+\kappa +\frac{1}{2}\right) \right] \mathrm{M}_{\kappa ,\mu }\left( x\right)
\label{DkM_I3} \\
&&+\frac{2^{-2\mu }\,x^{\mu +1/2}}{\mathrm{B}\left( \mu +\kappa +\frac{1}{2}%
,\mu -\kappa +\frac{1}{2}\right) }I_{3}\left( \kappa ,\mu ;x\right)  \notag
\\
&=&\left[ \psi \left( \mu -\kappa +\frac{1}{2}\right) -\psi \left( \mu
+\kappa +\frac{1}{2}\right) \right] \mathrm{M}_{\kappa ,\mu }\left( x\right)
\label{DkM_I4} \\
&&+\frac{2^{-2\mu }\,x^{\mu +1/2}}{\mathrm{B}\left( \mu +\kappa +\frac{1}{2}%
,\mu -\kappa +\frac{1}{2}\right) }I_{4}\left( \kappa ,\mu ;x\right) ,  \notag
\end{eqnarray}%
where we have defined the following logarithmic integrals:\

\begin{definition}
\begin{eqnarray}
I_{3}\left( \kappa ,\mu ;x\right) &=&\int_{-1}^{1}e^{xt/2}\left( 1+t\right)
^{\mu -\kappa -1/2}\left( 1-t\right) ^{\mu +\kappa -1/2}\ln \left( \frac{1-t%
}{1+t}\right) dt,  \label{I3_def} \\
I_{4}\left( \kappa ,\mu ;x\right) &=&\int_{-1}^{1}e^{-xt/2}\left( 1+t\right)
^{\mu +\kappa -1/2}\left( 1-t\right) ^{\mu -\kappa -1/2}\ln \left( \frac{1+t%
}{1-t}\right) dt.  \label{I4_def}
\end{eqnarray}
\end{definition}

Note that, from (\ref{DkM_I3})-(\ref{DkM_I4}), we have%
\begin{equation}
\,I_{3}\left( \kappa ,\mu ;x\right) =I_{4}\left( \kappa ,\mu ;x\right)
=2^{2\mu }e^{-x/2}I_{1}\left( \kappa ,\mu ;x\right) .  \label{I4=I3->I1}
\end{equation}

Since $\,I_{2}\left( \kappa ,\mu ;x\right) $, $\,I_{3}\left( \kappa ,\mu
;x\right) $, and $\,I_{4}\left( \kappa ,\mu ;x\right) $ are reduced to the
calculation of $\,I_{1}\left( \kappa ,\mu ;x\right) $, next we calculate the
latter integral.

\begin{theorem}
The following integral holds true for $x\in
\mathbb{R}
$:%
\begin{eqnarray}
&&I_{1}\left( \kappa ,\mu ;x\right)  \label{I1_general} \\
&=&\mathrm{B}\left( \mu +\kappa +\frac{1}{2},\mu -\kappa +\frac{1}{2}\right)
\notag \\
&&\left\{ \left[ \psi \left( \frac{1}{2}+\mu +\kappa \right) -\psi \left(
\frac{1}{2}+\mu -\kappa \right) \right] \,_{1}F_{1}\left( \left.
\begin{array}{c}
\frac{1}{2}+\mu -\kappa \\
1+2\mu%
\end{array}%
\right\vert x\right) \right.  \notag \\
&&-\left. G^{\left( 1\right) }\left( \left.
\begin{array}{c}
\frac{1}{2}+\mu -\kappa \\
1+2\mu%
\end{array}%
\right\vert x\right) \right\} .  \notag
\end{eqnarray}
\end{theorem}

\begin{proof}
Comparing (\ref{DkM_I1})\ to (\ref{DkM_G(1)}), taking into account (\ref%
{M_k,mu_def}), we arrive at (\ref{I1_general}), as we wanted to prove.
\end{proof}

\begin{corollary}
For $\kappa =0$, Eqn. (\ref{I1_general})\ is reduced to%
\begin{equation}
I_{1}\left( 0,\mu ;x\right) =-\mathrm{B}\left( \mu +\frac{1}{2},\mu +\frac{1%
}{2}\right) G^{\left( 1\right) }\left( \left.
\begin{array}{c}
\frac{1}{2}+\mu \\
1+2\mu%
\end{array}%
\right\vert x\right) .  \label{I1(0,mu,x)}
\end{equation}
\end{corollary}

\begin{theorem}
For $\ell \in
\mathbb{Z}
$ and $m=0,1,2,\ldots $, with $m\geq \ell $, the following integral holds
true for $x\in
\mathbb{R}
$:%
\begin{equation}
I_{1}\left( \frac{\ell }{2},m+\frac{1-\ell }{2};x\right) =e^{x}\mathcal{F}%
\left( -\ell ,m-\ell ,-x\right) -\mathcal{F}\left( \ell ,m,x\right) ,
\label{I1_ml_reduction}
\end{equation}%
where%
\begin{eqnarray}
&&\mathcal{F}\left( s,k,z\right)  \label{F_def} \\
&=&\sum_{n=0}^{k}\left( -1\right) ^{n}\binom{k}{n}\frac{d^{n+k-s}}{dz^{n+k-s}%
}\left[ \frac{\ln z-\mathrm{Chi}\left( z\right) -\mathrm{Shi}\left( z\right)
+\gamma }{z}\right] \,,  \notag
\end{eqnarray}%
and the functions $\mathrm{Shi}\left( z\right) $ and $\mathrm{Chi}\left(
z\right) $ denote the hyperbolic sine and cosine integrals.
\end{theorem}

\begin{proof}
From the definition of $I_{1}\left( \kappa ,\mu ;x\right) $ given in (\ref%
{I1_def}), we have%
\begin{eqnarray*}
I_{1}\left( \kappa ,\mu ;x\right) &=&\int_{0}^{1}e^{xt}t^{\mu -\kappa
-1/2}\left( 1-t\right) ^{\mu +\kappa -1/2}\ln \left( 1-t\right) dt \\
&&-\int_{0}^{1}e^{xt}t^{\mu -\kappa -1/2}\left( 1-t\right) ^{\mu +\kappa
-1/2}\ln t\,dt.
\end{eqnarray*}%
Perform the change of variables $\tau =1-t$ in the first integral above to
arrive at%
\begin{equation}
I_{1}\left( \kappa ,\mu ;x\right) =e^{x}\mathcal{I}_{1}\left( -\kappa ,\mu
;-x\right) -\mathcal{I}_{1}\left( \kappa ,\mu ;x\right) ,  \label{I1_F}
\end{equation}%
where we have set
\begin{equation}
\mathcal{I}_{1}\left( \kappa ,\mu ;x\right) =\int_{0}^{1}e^{xt}t^{\mu
-\kappa -1/2}\left( 1-t\right) ^{\mu +\kappa -1/2}\ln t\,dt.
\label{I1_calligrafic_def}
\end{equation}%
Taking into account the binomial theorem and the integral (\ref{gamma_int_2}%
) calculated in the Appendix, i.e.
\begin{equation*}
\int_{0}^{1}e^{xt}t^{m}\ln t\,dt=\frac{-1}{\left( m+1\right) ^{2}}%
\,_{2}F_{2}\left( \left.
\begin{array}{c}
m+1,m+1 \\
m+2,m+2%
\end{array}%
\right\vert x\right) ,
\end{equation*}%
calculate%
\begin{eqnarray}
&&\mathcal{I}_{1}\left( \frac{\ell }{2},m+\frac{1-\ell }{2};x\right)
\label{I1_calligrafic_resultado_a} \\
&=&\int_{0}^{1}e^{xt}t^{m-\ell }\left( 1-t\right) ^{m}\ln t\,dt  \notag \\
&=&\sum_{n=0}^{m}\binom{m}{n}\left( -1\right)
^{n}\int_{0}^{1}e^{xt}t^{m+n-\ell }\ln t\,dt  \notag \\
&=&\sum_{n=0}^{m}\binom{m}{n}\frac{\left( -1\right) ^{n+1}}{\left( n+m-\ell
+1\right) ^{2}}\,_{2}F_{2}\left( \left.
\begin{array}{c}
n+m-\ell +1,n+m-\ell +1 \\
n+m-\ell +2,n+m-\ell +2%
\end{array}%
\right\vert x\right) .  \notag
\end{eqnarray}%
Now, apply the differentiation formula \cite[Eqn. 16.3.1]{olver2010nist}%
\begin{equation*}
\frac{d^{n}}{dz^{n}}\,_{p}F_{q}\left( \left.
\begin{array}{c}
a_{1},\ldots ,a_{p} \\
b_{1},\ldots ,b_{q}%
\end{array}%
\right\vert z\right) =\frac{\left( a_{1}\right) _{n}\cdots \left(
a_{p}\right) _{n}}{\left( b_{1}\right) _{n}\cdots \left( b_{q}\right) _{n}}%
\,_{p}F_{q}\left( \left.
\begin{array}{c}
a_{1}+n,\ldots ,a_{p}+n \\
b_{1}+n,\ldots ,b_{q}+n%
\end{array}%
\right\vert z\right) ,
\end{equation*}%
to obtain%
\begin{equation}
\mathcal{I}_{1}\left( \frac{\ell }{2},m+\frac{1-\ell }{2};x\right)
=\sum_{n=0}^{m}\binom{m}{n}\left( -1\right) ^{n+1}\frac{d^{n+m-\ell }}{%
dx^{n+m-\ell }}\,_{2}F_{2}\left( \left.
\begin{array}{c}
1,1 \\
2,2%
\end{array}%
\right\vert x\right) .  \label{I1_calligrafic_D}
\end{equation}%
According to \cite[Eqn. 7.12.2(67)]{prudnikov1986integrals}, we have that%
\begin{equation}
_{2}F_{2}\left( \left.
\begin{array}{c}
1,1 \\
2,2%
\end{array}%
\right\vert x\right) =\frac{\mathrm{Ei}\left( x\right) -\ln \left( -x\right)
-\gamma }{x},  \label{2F2_Ei}
\end{equation}%
In order to obtain similar expressions as the ones obtained in Table \ref%
{Table_1}, we derive an alternative form of (\ref{2F2_Ei}). Indeed, from the
definition of the \textit{hyperbolic sine and cosine integrals} \cite[Eqns.
6.2.15-16]{olver2010nist}, $\forall z\in
\mathbb{C}
$,%
\begin{eqnarray*}
\mathrm{Shi}\left( z\right) &=&\int_{0}^{z}\frac{\sinh t}{t}dt \\
\mathrm{Chi}\left( z\right) &=&\gamma +\ln z+\int_{0}^{z}\frac{\cosh t-1}{t}%
dt,
\end{eqnarray*}%
it is easy to prove that%
\begin{eqnarray}
\mathrm{Shi}\left( -z\right) &=&-\mathrm{Shi}\left( z\right) ,
\label{Shi(-z)} \\
\mathrm{Chi}\left( -z\right) &=&\mathrm{Chi}\left( z\right) -\ln z+\ln
\left( -z\right) .  \label{Chi(-z)}
\end{eqnarray}%
Also, from the definition of \textit{complementary exponential integral}
\cite[Eqn. 6.2.3]{olver2010nist}:
\begin{equation*}
\mathrm{Ein}\left( z\right) =\int_{0}^{z}\frac{1-e^{-t}}{t}dt,
\end{equation*}%
and the property $\forall x>0$ \cite[Eqn. 6.2.7]{olver2010nist}%
\begin{equation*}
\mathrm{Ei}\left( -x\right) =-\mathrm{Ein}\left( x\right) +\ln x+\gamma ,
\end{equation*}%
it is easy to prove that%
\begin{equation*}
\mathrm{Ei}\left( -x\right) =\mathrm{Chi}\left( x\right) -\mathrm{Shi}\left(
x\right) ,
\end{equation*}%
thus, taking into account (\ref{Shi(-z)})\ and (\ref{Chi(-z)}), we have%
\begin{equation}
\mathrm{Ei}\left( x\right) =\mathrm{Chi}\left( x\right) -\ln x+\ln \left(
-x\right) +\mathrm{Shi}\left( x\right) .  \label{Ei(x)}
\end{equation}%
Insert (\ref{Ei(x)})\ in (\ref{2F2_Ei}), to obtain
\begin{equation}
_{2}F_{2}\left( \left.
\begin{array}{c}
1,1 \\
2,2%
\end{array}%
\right\vert x\right) =\frac{\mathrm{Chi}\left( x\right) -\ln x+\mathrm{Shi}%
\left( x\right) -\gamma }{x}.  \label{2F2_Shi_Chi}
\end{equation}%
Finally, substitute (\ref{2F2_Shi_Chi})\ in (\ref{I1_calligrafic_D}),\ and
take into account (\ref{F_def}), to arrive at%
\begin{eqnarray*}
&&\mathcal{I}_{1}\left( \frac{\ell }{2},m+\frac{1-\ell }{2};x\right) \\
&=&\sum_{n=0}^{m}\binom{m}{n}\left( -1\right) ^{n+1}\frac{d^{n+m-\ell }}{%
dx^{n+m-\ell }}\,\left[ \frac{\mathrm{Chi}\left( x\right) -\ln x+\mathrm{Shi}%
\left( x\right) -\gamma }{x}\right] \\
&=&\mathcal{F}\left( \ell ,m,x\right) .
\end{eqnarray*}%
Similarly, calculate%
\begin{equation}
\mathcal{I}_{1}\left( -\frac{\ell }{2},m+\frac{1-\ell }{2};-x\right) =%
\mathcal{F}\left( -\ell ,m-\ell ,-x\right) .
\label{I1_calligrafic_resultado_b}
\end{equation}%
Finally, according to (\ref{I1_F}), we arrive at (\ref{I1_ml_reduction}), as
we wanted to prove.
\end{proof}

Table \ref{Table_3A} shows the integral $I_{1}\left( \kappa ,\mu ;x\right) $
for $x\in
\mathbb{R}
$ and particular values of the parameters $\kappa $ and/or $\mu $, obtained
from (\ref{I1_general}) and (\ref{I1_ml_reduction}) with the aid of
MATHEMATICA\ program.

\begin{center}
\begin{table}[htbp] \centering%
\caption{Integral $I_1(\kappa,\mu;x)$  for particular values of $\kappa$ and
$\mu$.}%
\rotatebox{90}{
\begin{tabular}{|c|c|c|}
\hline
$\kappa $ & $\mu $ & $I_{1}\left( \kappa ,\mu ;x\right) $ \\ \hline\hline
$-\frac{1}{2}$ & $1$ & $\frac{1}{x^{2}}\left\{ e^{x}\left( 1-x\right) \left[
\ln x+\gamma +\mathrm{Shi}\left( x\right) -\mathrm{Chi}\left( x\right) %
\right] +\ln x+\gamma -\mathrm{Chi}\left( x\right) -\mathrm{Shi}\left(
x\right) \right\} $ \\ \hline
$-\frac{1}{2}$ & $\mu $ & $-\frac{\sqrt{\pi }}{2}\Gamma \left( \mu \right)
\left\{ \frac{e^{x/2}x^{1/2-\mu }}{\mu }\left[ I_{\mu -1/2}\left( \frac{x}{2}%
\right) +I_{\mu +1/2}\left( \frac{x}{2}\right) \right] +\frac{2^{1-2\mu }}{%
\Gamma \left( \mu +\frac{1}{2}\right) }G^{\left( 1\right) }\left( \mu
+1;2\mu +1;x\right) \right\} $ \\ \hline
$\frac{1}{2}$ & $1$ & $\frac{1}{x^{2}}\left\{ \left( x+e^{x}+1\right) \left[
\mathrm{Chi}\left( x\right) -\ln x-\gamma \right] +\left( x-e^{x}+1\right)
\mathrm{Shi}\left( x\right) \right\} $ \\ \hline
$\frac{1}{2}$ & $\mu $ & $\frac{\sqrt{\pi }}{2}\Gamma \left( \mu \right)
\left\{ \frac{e^{x/2}x^{1/2-\mu }}{\mu }\left[ I_{\mu -1/2}\left( \frac{x}{2}%
\right) -I_{\mu +1/2}\left( \frac{x}{2}\right) \right] -\frac{2^{1-2\mu }}{%
\Gamma \left( \mu +\frac{1}{2}\right) }G^{\left( 1\right) }\left( \mu ;2\mu
+1;x\right) \right\} $ \\ \hline
$1$ & $\mu $ & $%
\begin{array}{l}
\Gamma \left( \mu -\frac{1}{2}\right) \left\{ \frac{4\sqrt{\pi }\mu
\,e^{x/2}x^{-\mu }}{4\mu ^{2}-1}\left[ \left( 2\mu -x+1\right) I_{\mu
}\left( \frac{x}{2}\right) +x\,I_{\mu +1}\left( \frac{x}{2}\right) \right]
\right. \\
\quad \left. -\frac{\Gamma \left( \mu +\frac{2}{3}\right) }{\Gamma \left(
2\mu +1\right) }G^{\left( 1\right) }\left( \mu -\frac{1}{2};2\mu +1;x\right)
\right\}%
\end{array}%
$ \\ \hline
$\kappa $ & $0$ & $\pi \sec \left( \pi \,\kappa \right) \left[ \pi \,\tan
\left( \pi \,\kappa \right) L_{\kappa -1/2}\left( x\right) -G^{\left(
1\right) }\left( \frac{1}{2}-\kappa ;1;x\right) \right] $ \\ \hline
$\kappa $ & $\frac{1}{2}$ & $-\pi \csc \left( \pi \,\kappa \right) \left\{ %
\left[ \pi \,\kappa \cot \left( \pi \,\kappa \right) -1\right]
\,_{1}F_{1}\left( 1-\kappa ;2;x\right) +\kappa \,G^{\left( 1\right) }\left(
1-\kappa ;2;x\right) \right\} $ \\ \hline
$\kappa $ & $\kappa $ & $\sqrt{\pi }\frac{\Gamma \left( 2\kappa +\frac{1}{2}%
\right) }{\Gamma \left( 2\kappa +1\right) }\left\{ \left[ H_{2\kappa
-1/2}+2\ln 2\right] \,_{1}F_{1}\left( \frac{1}{2};2\kappa +1;x\right)
-G^{\left( 1\right) }\left( \frac{1}{2};2\kappa +1;x\right) \right\} $ \\
\hline
$\frac{1}{4}$ & $\frac{1}{4}$ & $\frac{4e^{x}\ln 2}{\sqrt{x}}F\left( \sqrt{x}%
\right) -2G^{\left( 1\right) }\left( \frac{1}{2};\frac{3}{2};x\right) $ \\
\hline
\end{tabular}%
}
\label{Table_3A}%
\end{table}%
\end{center}

\begin{theorem}
For $\ell \in
\mathbb{Z}
$ and $m=0,1,2,\ldots $, with $m\geq \ell $, the following reduction formula
holds true for $x\in
\mathbb{R}
$:%
\begin{eqnarray}
&&\mathrm{M}_{\ell /2,m+\left( 1-\ell \right) /2}\left( x\right)
\label{M_ml_resultado} \\
&=&\left( 2m-\ell +1\right) \binom{2m-\ell }{m}\left( -1\right) ^{m-\ell
}x^{\ell /2-m}  \notag \\
&&\left[ e^{x/2}\mathcal{P}\left( -\ell ,m-\ell ,-x\right) -e^{-x/2}\mathcal{%
P}\left( \ell ,m,x\right) \right] ,  \notag
\end{eqnarray}%
where we have set the polynomials:%
\begin{equation}
\mathcal{P}\left( s,k,z\right) =\sum_{n=0}^{k}\binom{k}{n}\left(
2k-s-n\right) !\,z^{n}.  \label{M_sum_def}
\end{equation}
\end{theorem}

\begin{proof}
According to the definition of $\mathrm{M}_{\kappa ,\mu }\left( x\right) $ (%
\ref{M_k,mu_def}), we have%
\begin{equation}
\mathrm{M}_{\ell /2,m+\left( 1-\ell \right) /2}\left( x\right) =x^{m+1-\ell
/2}e^{-x/2}\,_{1}F_{1}\left( \left.
\begin{array}{c}
m+1-\ell \\
2\left( m+1\right) -\ell%
\end{array}%
\right\vert x\right) .  \label{M_ml}
\end{equation}%
Applying the property \cite[Eqn. 18:5:1]{oldham2009atlas}
\begin{equation*}
\left( -x\right) _{n}=\left( -1\right) ^{n}\left( x-n+1\right) _{n},
\end{equation*}%
and the reduction formula \cite[Eqn. 7.11.1(12)]{prudnikov1986integrals}
\begin{eqnarray*}
&&_{1}F_{1}\left( \left.
\begin{array}{c}
n \\
m%
\end{array}%
\right\vert z\right) =\frac{\left( m-2\right) !\left( 1-m\right) _{n}}{%
\left( n-1\right) !}\,z^{1-m} \\
&&\quad \left\{ \sum_{k=0}^{m-n-1}\frac{\left( 1+n-m\right) _{k}}{k!\left(
2-m\right) _{k}}\,z^{k}-e^{z}\sum_{k=0}^{n-1}\frac{\left( 1-n\right) _{k}}{%
k!\left( 2-m\right) _{k}}\left( -z\right) ^{k}\right\} ,
\end{eqnarray*}%
where $n,m=1,2,\ldots $ and $m>n$, after some algebra, we arrive at%
\begin{eqnarray}
&&_{1}F_{1}\left( \left.
\begin{array}{c}
m+1-\ell \\
2\left( m+1\right) -\ell%
\end{array}%
\right\vert x\right)  \notag \\
&=&\left( 2m-\ell +1\right) \binom{2m-\ell }{m}\left( -1\right) ^{m+1-\ell
}x^{\ell -2m}  \label{1F1_ml} \\
&&\quad \left\{ \sum_{k=0}^{m}\binom{m}{k}\,\left( 2m-\ell -k\right)
!x^{k}-e^{x}\sum_{k=0}^{m-\ell }\binom{m-\ell }{k}\,\left( 2m-\ell -k\right)
!\left( -x\right) ^{k}\right\} .  \notag
\end{eqnarray}%
Insert (\ref{1F1_ml})\ in (\ref{M_ml})\ to obtain (\ref{M_ml_resultado}), as
we wanted to prove.
\end{proof}

In addition to (\ref{M_ml_resultado}), other reduction formulas for the
Whittaker function $\mathrm{M}_{\kappa ,\mu }\left( x\right) $ are presented
in Appendix \ref{Appendix_Reduction_Whittaker}. A large list of reduction
formulas for $\mathrm{M}_{\kappa ,\mu }\left( x\right) $ is available in
\cite{apelblat2021integral} and in other monographs dealing with the special
functions \cite%
{erdelyi1953bateman,slater1960confluent,whittaker1920course,olver2010nist,magnus2013formulas,buchholz1969confluent,gradstein2015table,prudnikov1986integrals,oldham2009atlas,brychkov2008handbook}%
.

\begin{theorem}
For $\ell \in
\mathbb{Z}
$ and $m=0,1,2,\ldots $, with $m\geq \ell $, the following reduction formula
holds true for $x\in
\mathbb{R}
$:%
\begin{eqnarray}
&&\left. \frac{\partial \mathrm{M}_{\kappa ,\mu }\left( x\right) }{\partial
\kappa }\right\vert _{\kappa =\ell /2,\mu =m+\left( 1-\ell \right) /2}
\label{DkMlm} \\
&=&\left( 2m-\ell +1\right) \binom{2m-\ell }{m}x^{\ell /2-m}e^{-x/2}  \notag
\\
&&\left\{ \left( -1\right) ^{m-\ell }\left( H_{m-\ell }-H_{m}\right) \left[
e^{x}\mathcal{P}\left( -\ell ,m-\ell ,-x\right) -\mathcal{P}\left( \ell
,m,x\right) \right] \right.  \notag \\
&&+\left. x^{2m+1-\ell }\left[ e^{x}\mathcal{F}\left( -\ell ,m-\ell
,-x\right) -\mathcal{F}\left( \ell ,m,x\right) \right] \right\} .  \notag
\end{eqnarray}
\end{theorem}

\begin{proof}
According to (\ref{DkM_I1}), we have%
\begin{eqnarray*}
&&\left. \frac{\partial \mathrm{M}_{\kappa ,\mu }\left( x\right) }{\partial
\kappa }\right\vert _{\kappa =\ell /2,\mu =m+\left( 1-\ell \right) /2} \\
&=&\left[ \psi \left( m-\ell +1\right) -\psi \left( m+1\right) \right]
\mathrm{M}_{\ell /2,m+\left( 1-\ell \right) /2}\left( x\right) \\
&&+\frac{\,x^{m+1+\ell /2}e^{-x/2}}{\mathrm{B}\left( m+1,m-\ell +1\right) }%
I_{1}\left( \frac{\ell }{2},m+\frac{1-\ell }{2};x\right) .
\end{eqnarray*}%
Now, apply (\ref{Beta_def})\ and the property (\ref{Psi(n+1)})\ to get
\begin{eqnarray*}
&&\left. \frac{\partial \mathrm{M}_{\kappa ,\mu }\left( x\right) }{\partial
\kappa }\right\vert _{\kappa =\ell /2,\mu =m+\left( 1-\ell \right) /2} \\
&=&\left( H_{m-\ell }-H_{m}\right) \mathrm{M}_{\ell /2,m+\left( 1-\ell
\right) /2}\left( x\right) \\
&&+\left( 2m-\ell +1\right) \binom{2m-\ell }{m}x^{m+1-\ell
/2}e^{-x/2}I_{1}\left( \frac{\ell }{2},m+\frac{1-\ell }{2};x\right) .
\end{eqnarray*}%
Finally, applying the results given in (\ref{I1_ml_reduction})\ and (\ref%
{M_ml_resultado}), we arrive at (\ref{DkMlm}), as we wanted to prove.
\end{proof}

\begin{corollary}
For $\ell \in
\mathbb{Z}
$ and $m=0,1,2,\ldots $, with $m\geq \ell $, the following reduction formula
holds true for $x\in
\mathbb{R}
$:%
\begin{eqnarray}
&&G^{\left( 1\right) }\left( \left.
\begin{array}{c}
m+1-\ell \\
2\left( m+1\right) -\ell%
\end{array}%
\right\vert x\right)  \label{G(1)_ml_resultado} \\
&=&\left( 2m-\ell +1\right) \binom{2m-\ell }{m}  \notag \\
&&\left\{ \left( -1\right) ^{m-\ell }x^{\ell -2m-1}\left( H_{m-\ell
}-H_{m}\right) \left[ \mathcal{P}\left( \ell ,m,x\right) -e^{x}\mathcal{P}%
\left( -\ell ,m-\ell ,-x\right) \right] \right.  \notag \\
&&+\left. \mathcal{F}\left( \ell ,m,x\right) -e^{x}\mathcal{F}\left( -\ell
,m-\ell ,-x\right) \right\} .  \notag
\end{eqnarray}
\end{corollary}

\begin{proof}
Set (\ref{DkM_G(1)})\ for $\kappa =\frac{\ell }{2}$ and $\mu =m+\frac{1-\ell
}{2}$ and compare the result to (\ref{DkMlm}).
\end{proof}

Table \ref{Table_3}\ shows the first derivative of $\mathrm{M}_{\kappa ,\mu
}\left( x\right) $ with respect to parameter $\kappa $ for some particular
values of $\kappa $ and $\mu $, and $x\in
\mathbb{R}
$, which has been calculated from (\ref{DkMlm}) and are not contained in
Table \ref{Table_1}.

\begin{center}
\begin{table}[htbp] \centering%
\caption{Derivative of $M_{\kappa,\mu}$ with respect
to $\kappa$ by using (\ref{DkMlm}).}%
\rotatebox{90}{
\begin{tabular}{|c|c|c|}
\hline
$\kappa $ & $\mu $ & $\frac{\partial \mathrm{M}_{\kappa ,\mu }\left(
x\right) }{\partial \kappa }$ \\ \hline\hline
$-\frac{3}{2}$ & $2$ & $%
\begin{array}{l}
-\frac{4}{x^{3/2}}\left\{ e^{x/2}\left[ \left( x^{3}-3x^{2}+6x-6\right)
\left( \mathrm{Shi}\left( x\right) -\mathrm{Chi}\left( x\right) +\ln
x+\gamma \right) -\frac{11}{6}x^{3}+\frac{15}{2}x^{2}-15x+11\right] \right.
\\
\quad +\left. e^{-x/2}\left[ 6\left( \mathrm{Chi}\left( x\right) +\mathrm{Shi%
}\left( x\right) -\ln x-\gamma \right) -x^{2}+4x-11\right] \right\}%
\end{array}%
$ \\ \hline
$-1$ & $\frac{3}{2}$ & $%
\begin{array}{l}
\frac{3}{2x}\left\{ e^{x/2}\left[ \left( 2x^{2}-4x+4\right) \left( \mathrm{%
Chi}\left( x\right) -\mathrm{Shi}\left( x\right) -\ln x-\gamma \right)
+3x^{2}-8x+6\right] \right. \\
\quad +\left. 2\,e^{-x/2}\left[ 2\,\mathrm{Chi}\left( x\right) +2\,\mathrm{%
Shi}\left( x\right) +x-2\ln x-2\gamma -3\right] \right\}%
\end{array}%
$ \\ \hline
$-\frac{1}{2}$ & $1$ & $%
\begin{array}{l}
\frac{2}{\sqrt{x}}\left\{ e^{x/2}\left( x-1\right) \left( \mathrm{Chi}\left(
x\right) -\mathrm{Shi}\left( x\right) -\ln x-\gamma +1\right) \right. \\
\quad +\left. e^{-x/2}\left( \ln x-\mathrm{Chi}\left( x\right) -\mathrm{Shi}%
\left( x\right) +\gamma +1\right) \right\}%
\end{array}%
$ \\ \hline
$-\frac{1}{2}$ & $2$ & $%
\begin{array}{l}
\frac{6}{x^{3/2}}\left\{ e^{x/2}\left[ \left( x^{2}-4x+6\right) \left( 2\,%
\mathrm{Chi}\left( x\right) -2\,\mathrm{Shi}\left( x\right) -2\ln x-2\gamma
+3\right) -12\right] \right. \\
\quad +\left. e^{-x/2}\left[ 6\left( x-1\right) -4\left( x+3\right) \left(
\ln x-\mathrm{Chi}\left( x\right) -\mathrm{Shi}\left( x\right) +\gamma
\right) \right] \right\}%
\end{array}%
$ \\ \hline
$0$ & $\frac{3}{2}$ & $%
\begin{array}{l}
\frac{6}{x}\left\{ e^{x/2}\left[ \left( x-2\right) \left( \mathrm{Chi}\left(
x\right) -\mathrm{Shi}\left( x\right) -\ln x-\gamma \right) +x\right] \right.
\\
\quad +\left. e^{-x/2}\left[ \left( x+2\right) \left( \ln x-\mathrm{Chi}%
\left( x\right) -\mathrm{Shi}\left( x\right) +\gamma \right) -x\right]
\right\}%
\end{array}%
$ \\ \hline
$\frac{1}{2}$ & $2$ & $%
\begin{array}{l}
\frac{6}{x^{3/2}}\left\{ e^{x/2}\left[ 6\left( x+1\right) -4\left(
x-3\right) \left( \ln x+\mathrm{Shi}\left( x\right) -\mathrm{Chi}\left(
x\right) +\gamma \right) \right] \right. \\
\quad +\left. e^{-x/2}\left[ \left( x^{2}+4x+6\right) \left( 2\ln x-2\,%
\mathrm{Chi}\left( x\right) -2\,\mathrm{Shi}\left( x\right) +2\gamma
-3\right) +12\right] \right\}%
\end{array}%
$ \\ \hline
\end{tabular}%
}
\label{Table_3}%
\end{table}%
\end{center}

\subsection{Application to the calculation of infinite integrals}

Additional integral representations of the Whittaker function $\mathrm{M}%
_{\kappa ,\mu }\left( x\right) $ in terms of Bessel functions \cite[Sect.
6.5.1]{magnus2013formulas} are known:%
\begin{eqnarray}
&&\mathrm{M}_{\kappa ,\mu }\left( x\right)  \notag \\
&=&\frac{\Gamma \left( 1+2\mu \right) \,x^{1/2}e^{-x/2}}{\Gamma \left( \mu
-\kappa +\frac{1}{2}\right) }\int_{0}^{\infty }e^{-t}t^{-\kappa -1/2}I_{2\mu
}\left( 2\sqrt{xt}\right) dt  \label{Int_M_BesselI} \\
&=&\frac{\Gamma \left( 1+2\mu \right) \,x^{1/2}e^{x/2}}{\Gamma \left( \mu
+\kappa +\frac{1}{2}\right) }\int_{0}^{\infty }e^{-t}t^{\kappa -1/2}J_{2\mu
}\left( 2\sqrt{xt}\right) dt  \label{Int_M_BesselJ} \\
&&\mathrm{Re}\left( -\frac{1}{2}-\mu +\kappa \right) >0.  \notag
\end{eqnarray}

Let us introduce the following infinite logarithmic integrals.

\begin{definition}
\begin{eqnarray}
\mathcal{H}_{1}\left( \kappa ,\mu ;x\right) &=&\int_{0}^{\infty
}e^{-t}t^{-\kappa -1/2}I_{2\mu }\left( 2\sqrt{xt}\right) \ln t\,dt,
\label{H1_def} \\
\mathcal{H}_{2}\left( \kappa ,\mu ;x\right) &=&\int_{0}^{\infty
}e^{-t}t^{\kappa -1/2}J_{2\mu }\left( 2\sqrt{xt}\right) \ln t\,dt.
\label{H2_def}
\end{eqnarray}
\end{definition}

Differentiation of (\ref{Int_M_BesselI}) and (\ref{Int_M_BesselJ})\ with
respect to parameter $\kappa $ yields respectively%
\begin{eqnarray}
&&\frac{\partial \mathrm{M}_{\kappa ,\mu }\left( x\right) }{\partial \kappa }
\notag \\
&=&\psi \left( \mu -\kappa +\frac{1}{2}\right) \mathrm{M}_{\kappa ,\mu
}\left( x\right) -\frac{\Gamma \left( 1+2\mu \right) \,x^{1/2}e^{-x/2}}{%
\Gamma \left( \mu -\kappa +\frac{1}{2}\right) }\mathcal{H}_{1}\left( \kappa
,\mu ;x\right)  \label{DkM_H1} \\
&=&-\psi \left( \mu +\kappa +\frac{1}{2}\right) \mathrm{M}_{\kappa ,\mu
}\left( x\right) +\frac{\Gamma \left( 1+2\mu \right) \,x^{1/2}e^{x/2}}{%
\Gamma \left( \mu +\kappa +\frac{1}{2}\right) }\mathcal{H}_{2}\left( \kappa
,\mu ;x\right) .  \label{DkM_H2}
\end{eqnarray}

Note that, from (\ref{DkM_I1})\ and (\ref{DkM_H1})\ we have%
\begin{eqnarray}
&&\mathcal{H}_{1}\left( \kappa ,\mu ;x\right)  \label{H1->I1} \\
&=&\frac{\Gamma \left( \mu -\kappa +\frac{1}{2}\right) \psi \left( \mu
+\kappa +\frac{1}{2}\right) }{\Gamma \left( 1+2\mu \right) \sqrt{x}e^{x/2}}%
\mathrm{M}_{\kappa ,\mu }\left( x\right) -\frac{x^{\mu }\,I_{1}\left( \kappa
,\mu ;x\right) }{\Gamma \left( \mu +\kappa +\frac{1}{2}\right) },  \notag
\end{eqnarray}%
and from (\ref{DkM_I1})\ and (\ref{DkM_H2}),
\begin{eqnarray}
&&\mathcal{H}_{2}\left( \kappa ,\mu ;x\right)  \label{H2->I1} \\
&=&\frac{\Gamma \left( \mu +\kappa +\frac{1}{2}\right) \psi \left( \mu
-\kappa +\frac{1}{2}\right) }{\Gamma \left( 1+2\mu \right) \sqrt{x}e^{x/2}}%
\mathrm{M}_{\kappa ,\mu }\left( x\right) +\frac{e^{-x}x^{\mu }\,I_{1}\left(
\kappa ,\mu ;x\right) }{\Gamma \left( \mu -\kappa +\frac{1}{2}\right) }.
\notag
\end{eqnarray}

\begin{corollary}
For $\ell \in
\mathbb{Z}
$ and $m=0,1,2,\ldots $, with $m\geq \ell $, the following infinite
integrals holds true for $x\in
\mathbb{R}
$:%
\begin{eqnarray}
&&\int_{0}^{\infty }\frac{e^{-t}\ln t}{t^{\left( 1+\ell \right) /2}}%
\,I_{2m+1-\ell }\left( 2\sqrt{xt}\right) \,dt  \label{H1lm_resultado} \\
&=&\mathcal{H}_{1}\left( \frac{\ell }{2},m+\frac{1-\ell }{2};x\right)  \notag
\\
&=&\frac{1}{m!}\left\{ \left( -1\right) ^{m-\ell }\left( H_{m}-\gamma
\right) x^{-m+\left( \ell -1\right) /2}\left[ e^{x}\mathcal{P}\left( -\ell
,m-\ell ,-x\right) -\mathcal{P}\left( \ell ,m,x\right) \right] \right.
\notag \\
&&\quad -\left. x^{m+\left( 1-\ell \right) /2}\left[ e^{x}\mathcal{F}\left(
-\ell ,m-\ell ,-x\right) -\mathcal{F}\left( \ell ,m,x\right) \right]
\right\} .  \notag
\end{eqnarray}%
and%
\begin{eqnarray}
&&\int_{0}^{\infty }\frac{e^{-t}\ln t}{t^{\left( 1-\ell \right) /2}}%
J_{2m+1-\ell }\left( 2\sqrt{xt}\right) dt  \label{H2lm_resultado} \\
&=&\mathcal{H}_{2}\left( \frac{\ell }{2},m+\frac{1-\ell }{2};x\right)  \notag
\\
&=&\frac{1}{\left( m-\ell \right) !}\left\{ \left( -1\right) ^{m-\ell
}\left( H_{m-\ell }-\gamma \right) x^{-m+\left( \ell -1\right) /2}\left[
\mathcal{P}\left( -\ell ,m-\ell ,-x\right) -e^{-x}\mathcal{P}\left( \ell
,m,x\right) \right] \right.  \notag \\
&&\quad +\left. x^{m+\left( 1-\ell \right) /2}\left[ \mathcal{F}\left( -\ell
,m-\ell ,-x\right) -e^{-x}\mathcal{F}\left( \ell ,m,x\right) \right]
\right\} .  \notag
\end{eqnarray}
\end{corollary}

\begin{proof}
Substitute in (\ref{H1->I1})\ and (\ref{H2->I1}) the results given in (\ref%
{I1_ml_reduction})\ and (\ref{M_ml_resultado}), and apply (\ref{Psi(n+1)}).
\end{proof}

\subsection{Derivative with respect to the second parameter $\partial
\mathrm{M}_{\protect\kappa ,\protect\mu }\left( x\right) /\partial \protect%
\mu $}

In order to calculate the first derivative of $\mathrm{M}_{\kappa ,\mu
}\left( x\right) $ with respect to parameter $\mu $, let us introduce the
following finite logarithmic integrals.

\begin{definition}
\begin{eqnarray}
J_{1}\left( \kappa ,\mu ;x\right) &=&\int_{0}^{1}e^{xt}t^{\mu -\kappa
-1/2}\left( 1-t\right) ^{\mu +\kappa -1/2}\ln \left[ t\left( 1-t\right) %
\right] dt,  \label{J1_def} \\
J_{2}\left( \kappa ,\mu ;x\right) &=&\int_{0}^{1}e^{-xt}t^{\mu +\kappa
-1/2}\left( 1-t\right) ^{\mu -\kappa -1/2}\ln \left[ t\left( 1-t\right) %
\right] dt,  \label{J2_def} \\
J_{3}\left( \kappa ,\mu ;x\right) &=&\int_{-1}^{1}e^{xt/2}\left( 1+t\right)
^{\mu -\kappa -1/2}\left( 1-t\right) ^{\mu +\kappa -1/2}\ln \left(
1-t^{2}\right) dt,  \label{J3_def} \\
J_{4}\left( \kappa ,\mu ;x\right) &=&\int_{-1}^{1}e^{-xt/2}\left( 1+t\right)
^{\mu +\kappa -1/2}\left( 1-t\right) ^{\mu -\kappa -1/2}\ln \left(
1-t^{2}\right) dt.  \label{J4_def}
\end{eqnarray}
\end{definition}

Differentiation of (\ref{M_k,mu_int_1}) and (\ref{M_k,mu_int_2})\ with
respect to parameter $\mu $ gives%
\begin{eqnarray}
&&\frac{\partial \mathrm{M}_{\kappa ,\mu }\left( x\right) }{\partial \mu }
\notag \\
&=&\left[ \ln x-\psi \left( \mu -\kappa +\frac{1}{2}\right) -\psi \left( \mu
+\kappa +\frac{1}{2}\right) +2\,\psi \left( 2\mu +1\right) \right] \mathrm{M}%
_{\kappa ,\mu }\left( x\right)  \notag \\
&&+\frac{x^{\mu +1/2}e^{-x/2}}{\mathrm{B}\left( \mu +\kappa +\frac{1}{2},\mu
-\kappa +\frac{1}{2}\right) }J_{1}\left( \kappa ,\mu ;x\right)
\label{DmM_J1} \\
&=&\left[ \ln x-\psi \left( \mu -\kappa +\frac{1}{2}\right) -\psi \left( \mu
+\kappa +\frac{1}{2}\right) +2\,\psi \left( 2\mu +1\right) \right] \mathrm{M}%
_{\kappa ,\mu }\left( x\right)  \notag \\
&&+\frac{x^{\mu +1/2}e^{x/2}}{\mathrm{B}\left( \mu +\kappa +\frac{1}{2},\mu
-\kappa +\frac{1}{2}\right) }J_{2}\left( \kappa ,\mu ;x\right).
\label{DmM_J2}
\end{eqnarray}

For the other integral representations given in (\ref{M_k,mu_int_3})\ and (%
\ref{M_k,mu_int_4}), we have%
\begin{eqnarray}
&&\frac{\partial \mathrm{M}_{\kappa ,\mu }\left( x\right) }{\partial \mu }
\notag \\
&=&\left[ \ln \left( x/4\right) -\psi \left( \mu -\kappa +\frac{1}{2}\right)
-\psi \left( \mu +\kappa +\frac{1}{2}\right) +2\,\psi \left( 2\mu +1\right) %
\right] \mathrm{M}_{\kappa ,\mu }\left( x\right)  \notag \\
&&+\frac{2^{-2\mu }\,x^{\mu +1/2}}{\mathrm{B}\left( \mu +\kappa +\frac{1}{2}%
,\mu -\kappa +\frac{1}{2}\right) }J_{3}\left( \kappa ,\mu ;x\right)
\label{DmM_J3} \\
&=&\left[ \ln \left( x/4\right) -\psi \left( \mu -\kappa +\frac{1}{2}\right)
-\psi \left( \mu +\kappa +\frac{1}{2}\right) +2\,\psi \left( 2\mu +1\right) %
\right] \mathrm{M}_{\kappa ,\mu }\left( x\right)  \notag \\
&&+\frac{2^{-2\mu }\,x^{\mu +1/2}}{\mathrm{B}\left( \mu +\kappa +\frac{1}{2}%
,\mu -\kappa +\frac{1}{2}\right) }J_{4}\left( \kappa ,\mu ;x\right).
\label{DmM_J4}
\end{eqnarray}

From (\ref{DmM_J1})-(\ref{DmM_J4}), we obtain the following
interrelationships:\
\begin{eqnarray*}
J_{2}\left( \kappa ,\mu ;x\right) &=&e^{-x}J_{1}\left( \kappa ,\mu ;x\right)
, \\
J_{3}\left( \kappa ,\mu ;x\right) &=&2^{2\mu }\left[ e^{-x/2}J_{1}\left(
\kappa ,\mu ;x\right) +\frac{\ln 4}{x^{\mu +1/2}}\mathrm{B}\left( \mu
+\kappa +\frac{1}{2},\mu -\kappa +\frac{1}{2}\right) \,\mathrm{M}_{\kappa
,\mu }\left( x\right) \right] , \\
J_{4}\left( \kappa ,\mu ;x\right) &=&J_{3}\left( \kappa ,\mu ;x\right) .
\end{eqnarray*}

Since $\,J_{2}\left( \kappa ,\mu ;x\right) $, $\,J_{3}\left( \kappa ,\mu
;x\right) $, and $\,J_{4}\left( \kappa ,\mu ;x\right) $ are reduced to the
calculation of $\,J_{1}\left( \kappa ,\mu ;x\right) $, next we calculate the
latter integral.

\begin{theorem}
According to the notation introduced in (\ref{G(1)_def})\ and (\ref{H(1)_def}%
), the following integral holds true:%
\begin{eqnarray}
&&J_{1}\left( \kappa ,\mu ;x\right)  \label{J1_general} \\
&=&\mathrm{B}\left( \mu +\kappa +\frac{1}{2},\mu -\kappa +\frac{1}{2}\right)
\notag \\
&&\left\{ \left[ \psi \left( \frac{1}{2}+\mu +\kappa \right) +\psi \left(
\frac{1}{2}+\mu -\kappa \right) -2\,\psi \left( 2\mu +1\right) \right]
\,_{1}F_{1}\left( \left.
\begin{array}{c}
\frac{1}{2}+\mu -\kappa \\
1+2\mu%
\end{array}%
\right\vert x\right) \right.  \notag \\
&&\left. +G^{\left( 1\right) }\left( \left.
\begin{array}{c}
\frac{1}{2}+\mu -\kappa \\
1+2\mu%
\end{array}%
\right\vert x\right) +2\,H^{\left( 1\right) }\left( \left.
\begin{array}{c}
\frac{1}{2}+\mu -\kappa \\
1+2\mu%
\end{array}%
\right\vert x\right) \right\} .  \notag
\end{eqnarray}
\end{theorem}

\begin{proof}
Comparing (\ref{DmM_J1})\ to (\ref{DmM_H(1)}), taking into account (\ref%
{M_k,mu_def}), we arrive at (\ref{J1_general}), as we wanted to prove.
\end{proof}

\begin{theorem}
For $\ell \in
\mathbb{Z}
$ and $m=0,1,2,\ldots $, with $m\geq \ell $, the following integral holds
true for $x\in
\mathbb{R}
$:%
\begin{equation}
J_{1}\left( \frac{\ell }{2},m+\frac{1-\ell }{2};x\right) =e^{x}\mathcal{F}%
\left( -\ell ,m-\ell ,-x\right) +\mathcal{F}\left( \ell ,m,x\right) .
\label{J1_ml_resultado}
\end{equation}
\end{theorem}

\begin{proof}
From the definition of $J_{1}\left( \kappa ,\mu ;x\right) $ given in (\ref%
{J1_def}), we have%
\begin{eqnarray*}
J_{1}\left( \kappa ,\mu ;x\right) &=&\int_{0}^{1}e^{xt}t^{\mu -\kappa
-1/2}\left( 1-t\right) ^{\mu +\kappa -1/2}\ln t\,dt \\
&&+\int_{0}^{1}e^{xt}t^{\mu -\kappa -1/2}\left( 1-t\right) ^{\mu +\kappa
-1/2}\ln \left( 1-t\right) \,dt.
\end{eqnarray*}%
Perform the change of variables $\tau =1-t$ in the second integral above to
arrive at%
\begin{equation}
J_{1}\left( \kappa ,\mu ;x\right) =e^{x}\mathcal{I}_{1}\left( -\kappa ,\mu
;-x\right) +\mathcal{I}_{1}\left( \kappa ,\mu ;x\right) ,
\label{J1->I1_calligrafic}
\end{equation}%
where we follow the notation given in (\ref{I1_calligrafic_def}) for the
integral $\mathcal{I}_{1}\left( \kappa ,\mu ;x\right) $. According to the
results obtained in (\ref{I1_calligrafic_resultado_a})\ and (\ref%
{I1_calligrafic_resultado_b}), we arrive at (\ref{J1_ml_resultado}), as we
wanted to prove.
\end{proof}

\begin{theorem}
For $\ell \in
\mathbb{Z}
$ and $m=0,1,2,\ldots $, with $m\geq \ell $, the following reduction formula
holds true for $x\in
\mathbb{R}
$:%
\begin{eqnarray}
&&\left. \frac{\partial \mathrm{M}_{\kappa ,\mu }\left( x\right) }{\partial
\mu }\right\vert _{\kappa =\ell /2,\mu =m+\left( 1-\ell \right) /2}
\label{DmM_lm_resultado} \\
&=&\left( 2m-\ell +1\right) \binom{2m-\ell }{m}x^{\ell /2-m}e^{-x/2}  \notag
\\
&&\left\{ \left( -1\right) ^{m-\ell }\left( \ln x+2\,H_{2m-\ell
+1}-H_{m-\ell }-H_{m}\right) \left[ e^{x}\mathcal{P}\left( -\ell ,m-\ell
,-x\right) -\mathcal{P}\left( \ell ,m,x\right) \right] \right.  \notag \\
&&+\left. x^{2m+1-\ell }\left[ e^{x}\mathcal{F}\left( -\ell ,m-\ell
,-x\right) +\mathcal{F}\left( \ell ,m,x\right) \right] \right\} .  \notag
\end{eqnarray}
\end{theorem}

\begin{proof}
Insert (\ref{M_ml_resultado})\ and (\ref{J1_ml_resultado})\ in (\ref{DmM_J1}%
)\ and apply (\ref{Psi(n+1)}).
\end{proof}

Table \ref{Table_4}\ shows the first derivative of $\mathrm{M}_{\kappa ,\mu
}\left( x\right) $ with respect to parameter $\mu $ for some particular
values of $\kappa $ and $\mu $, andfor $x\in
\mathbb{R}
$, which has been calculated from (\ref{DmM_lm_resultado}) and are not
contained in Tables \ref{Table_2A} and \ref{Table_2}.

\begin{center}
\begin{table}[htbp] \centering%
\caption{Derivative of $M_{\kappa,\mu}$ with respect
to $\mu$ by using (\ref{DmM_lm_resultado}).}%
\rotatebox{90}{
\begin{tabular}{|c|c|c|}
\hline
$\kappa $ & $\mu $ & $\frac{\partial \mathrm{M}_{\kappa ,\mu }\left(
x\right) }{\partial \mu }$ \\ \hline\hline
$-\frac{3}{2}$ & $2$ & $%
\begin{array}{l}
\frac{4}{x^{3/2}}\left\{ e^{x/2}\left[ \left( x^{3}-3x^{2}+6x-6\right)
\left( \mathrm{Chi}\left( x\right) -\mathrm{Shi}\left( x\right) -\gamma
\right) +\frac{7}{3}x^{3}-11x^{2}+28x-36\right] \right. \\
\quad +\left. e^{-x/2}\left[ 6\left( \mathrm{Chi}\left( x\right) +\mathrm{Shi%
}\left( x\right) -\gamma \right) +x^{2}-4x+36\right] \right\}%
\end{array}%
$ \\ \hline
$-1$ & $\frac{3}{2}$ & $%
\begin{array}{l}
\frac{1}{x}\left\{ e^{x/2}\left[ 3\left( x^{2}-2x+2\right) \left( \mathrm{Chi%
}\left( x\right) -\mathrm{Shi}\left( x\right) -\gamma \right) +\frac{13}{2}%
x^{2}-22x+31\right] \right. \\
\quad +\left. e^{-x/2}\left[ 3\left( x-2\,\mathrm{Chi}\left( x\right) -2\,%
\mathrm{Shi}\left( x\right) +2\gamma \right) -31\right] \right\}%
\end{array}%
$ \\ \hline
$-\frac{1}{2}$ & $1$ & $%
\begin{array}{l}
\frac{2}{\sqrt{x}}\left\{ e^{x/2}\left[ \left( x-1\right) \left( \mathrm{Chi}%
\left( x\right) -\mathrm{Shi}\left( x\right) -\gamma +2\right) -2\right]
\right. \\
\quad +\left. e^{-x/2}\left( \mathrm{Chi}\left( x\right) +\mathrm{Shi}\left(
x\right) -\gamma +4\right) \right\}%
\end{array}%
$ \\ \hline
$-\frac{1}{2}$ & $2$ & $%
\begin{array}{l}
\frac{8}{x^{3/2}}\left\{ e^{x/2}\left[ 3\left( \frac{1}{2}x^{2}-2x+3\right)
\left( \mathrm{Chi}\left( x\right) -\mathrm{Shi}\left( x\right) -\gamma
\right) +4x^{2}-22x+48\right] \right. \\
\quad -\left. e^{-x/2}\left[ 3\left( x+3\right) \left( \mathrm{Chi}\left(
x\right) +\mathrm{Shi}\left( x\right) -\gamma \right) +8\left( x+6\right) %
\right] \right\}%
\end{array}%
$ \\ \hline
$\frac{1}{2}$ & $1$ & $%
\begin{array}{l}
\frac{2}{\sqrt{x}}\left\{ e^{x/2}\left( \mathrm{Chi}\left( x\right) -\mathrm{%
Shi}\left( x\right) -\gamma +4\right) \right. \\
\quad -\left. e^{-x/2}\left[ \left( x+1\right) \left( \mathrm{Chi}\left(
x\right) +\mathrm{Shi}\left( x\right) -\gamma +2\right) +2\right] \right\}%
\end{array}%
$ \\ \hline
$\frac{1}{2}$ & $2$ & $%
\begin{array}{l}
\frac{4}{x^{3/2}}\left\{ e^{x/2}\left[ 6\left( x-3\right) \left( \mathrm{Chi}%
\left( x\right) -\mathrm{Shi}\left( x\right) -\gamma \right) +16\left(
x-6\right) \right] \right. \\
\quad +\left. e^{-x/2}\left[ 3\left( x^{2}+4x+6\right) \left( \mathrm{Chi}%
\left( x\right) +\mathrm{Shi}\left( x\right) -\gamma \right) +8x^{2}+44x+96%
\right] \right\}%
\end{array}%
$ \\ \hline
\end{tabular}%
}
\label{Table_4}%
\end{table}%
\end{center}

\begin{corollary}
For $\ell \in
\mathbb{Z}
$ and $m=0,1,2,\ldots $, with $m\geq \ell $, the following reduction formula
holds true for $x\in
\mathbb{R}
$:%
\begin{eqnarray}
&&H^{\left( 1\right) }\left( \left.
\begin{array}{c}
m+1-\ell \\
2\left( m+1\right) -\ell%
\end{array}%
\right\vert x\right)  \label{H(1)lm_resultado} \\
&=&\left( 2m-\ell +1\right) \binom{2m-\ell }{m}  \notag \\
&&\left\{ \left( -1\right) ^{m-\ell }x^{\ell -2m-1}\left( H_{2m-\ell
+1}-H_{m}\right) \left[ e^{x}\mathcal{P}\left( -\ell ,m-\ell ,-x\right) -%
\mathcal{P}\left( \ell ,m,x\right) \right] \right.  \notag \\
&&+\left. e^{x}\mathcal{F}\left( -\ell ,m-\ell ,-x\right) \right\} .  \notag
\end{eqnarray}
\end{corollary}

\begin{proof}
Take $\kappa =\frac{\ell }{2}$ and $\mu =m+\frac{1-\ell }{2}$ in (\ref%
{DmM_H(1)}), and substitute the results given in (\ref{M_ml_resultado}), (%
\ref{G(1)_ml_resultado})\ and (\ref{DmM_lm_resultado}). After
simplification, we arrive at (\ref{H(1)lm_resultado}), as we wanted to prove.
\end{proof}

\subsection{Application to the calculation of finite integrals}

\begin{theorem}
For $\mu \geq 0$ and $x\in
\mathbb{R}
$, the following finite integral holds true:%
\begin{eqnarray}
&&\int_{0}^{1}e^{xt}\left[ t\left( 1-t\right) \right] ^{\mu -1/2}\ln \left[
t\left( 1-t\right) \right] dt  \label{J1(0,mu;x)_resultado} \\
&=&J_{1}\left( 0,\mu ;x\right)  \notag \\
&=&\mathrm{B}\left( \mu +\frac{1}{2},\mu +\frac{1}{2}\right) \left( \frac{4}{%
\left\vert x\right\vert }\right) ^{\mu }e^{x/2}\Gamma \left( 1+\mu \right)
\notag \\
&&\left\{ I_{\mu }\left( \frac{\left\vert x\right\vert }{2}\right) \left[
\psi \left( \mu +\frac{1}{2}\right) -\ln \left\vert x\right\vert \right] +%
\frac{\partial I_{\mu }\left( \left\vert x\right\vert /2\right) }{\partial
\mu }\right\} ,  \notag
\end{eqnarray}%
where $\partial I_{\mu }\left( x\right) /\partial \mu $ is given by (\ref%
{DmuI_Meijer})\ or (\ref{DmuI_Hyper}).
\end{theorem}

\begin{proof}
First, consider that $x>0$. Take $\kappa =0$ in (\ref{DmM_J1}) and
susbtitute (\ref{M_0_mu}) to arrive at
\begin{eqnarray}
&&\left. \frac{\partial \mathrm{M}_{\kappa ,\mu }\left( x\right) }{\partial
\mu }\right\vert _{\kappa =0}  \label{DmMk=0_(1)} \\
&=&4^{\mu }\,\Gamma \left( 1+\mu \right) \sqrt{x}I_{\mu }\left( \frac{x}{2}%
\right) \left[ \ln x-2\,\psi \left( \mu +\frac{1}{2}\right) +2\,\psi \left(
2\mu +1\right) \right]  \notag \\
&&+\frac{x^{\mu +1/2}e^{-x/2}}{\mathrm{B}\left( \mu +\frac{1}{2},\mu +\frac{1%
}{2}\right) }J_{1}\left( 0,\mu ;x\right)  \notag
\end{eqnarray}%
Next, equate (\ref{DmMk=0_(1)}) to the expression given in (\ref{DmM0}), and
solve for $J_{1}\left( 0,\mu ;x\right) $ to get%
\begin{eqnarray}
&&J_{1}\left( 0,\mu ;x\right)  \label{DmMk=0_(2)} \\
&=&\mathrm{B}\left( \mu +\frac{1}{2},\mu +\frac{1}{2}\right) \left( \frac{4}{%
x}\right) ^{\mu }e^{x/2}\Gamma \left( 1+\mu \right)  \notag \\
&=&\left\{ I_{\mu }\left( \frac{x}{2}\right) \left[ \ln \left( \frac{4}{x}%
\right) +\psi \left( 1+\mu \right) +2\psi \left( \mu +\frac{1}{2}\right)
-2\psi \left( 2\mu +1\right) \right] +\frac{\partial I_{\mu }\left(
x/2\right) }{\partial \mu }\right\} .  \notag
\end{eqnarray}%
Now, apply the property \cite[Eqn. 5.5.8]{olver2010nist}
\begin{equation*}
\psi \left( 2z\right) =\frac{1}{2}\left[ \psi \left( z\right) +\psi \left( z+%
\frac{1}{2}\right) \right] +\ln 2,
\end{equation*}%
for $z=\mu +\frac{1}{2}$ to simplify (\ref{DmMk=0_(2)}) as
\begin{eqnarray}
&&J_{1}\left( 0,\mu ;x\right)  \label{J1(0,mu;x)_(1)} \\
&=&\mathrm{B}\left( \mu +\frac{1}{2},\mu +\frac{1}{2}\right) \left( \frac{4}{%
x}\right) ^{\mu }e^{x/2}\Gamma \left( 1+\mu \right)  \notag \\
&=&\left\{ I_{\mu }\left( \frac{x}{2}\right) \left[ \psi \left( \mu +\frac{1%
}{2}\right) -\ln x\right] +\frac{\partial I_{\mu }\left( x/2\right) }{%
\partial \mu }\right\} ,  \notag
\end{eqnarray}%
where (\ref{J1(0,mu;x)_(1)})\ holds true for $x>0$. Finally, note that
performing in (\ref{J1_def}) the change of variables $\tau =1-t$, we obtain
the reflection formula%
\begin{equation}
J_{1}\left( 0,\mu ;x\right) =e^{x}J_{1}\left( 0,\mu ;-x\right) ,
\label{J1_reflection}
\end{equation}%
so that from (\ref{J1(0,mu;x)_(1)})\ and (\ref{J1_reflection}) we arrive at (%
\ref{J1(0,mu;x)_resultado}), as we wanted to prove.
\end{proof}

\begin{theorem}
For $\mu \geq 0$ and $x\in
\mathbb{R}
$, the following finite integral holds true:%
\begin{eqnarray}
&&\int_{-1}^{1}e^{xt/2}\left[ t\left( 1-t\right) \right] ^{\mu -1/2}\ln %
\left[ t\left( 1-t\right) \right] dt  \label{J3(0,mu;x)_resultado} \\
&=&J_{3}\left( 0,\mu ;x\right)  \notag \\
&=&\mathrm{B}\left( \mu +\frac{1}{2},\mu +\frac{1}{2}\right) \Gamma \left(
1+\mu \right) \left( \frac{16}{\left\vert x\right\vert }\right) ^{\mu }
\notag \\
&&\left\{ I_{\mu }\left( \frac{\left\vert x\right\vert }{2}\right) \left[
\psi \left( \mu +\frac{1}{2}\right) +\ln \left( \frac{4}{\left\vert
x\right\vert }\right) \right] +\frac{\partial I_{\mu }\left( \left\vert
x\right\vert /2\right) }{\partial \mu }\right\} ,  \notag
\end{eqnarray}%
where $\partial I_{\mu }\left( x\right) /\partial \mu $ is given by (\ref%
{DmuI_Meijer})\ or (\ref{DmuI_Hyper}).
\end{theorem}

\begin{proof}
Consider $x>0$. Take $\kappa =0$ in (\ref{DmM_J3}) and susbtitute (\ref%
{M_0_mu}) to obtain%
\begin{eqnarray}
&&J_{3}\left( 0,\mu ;x\right)  \label{J3(0,mu;x)_(1)} \\
&=&2^{2\mu }e^{-x/2}J_{1}\left( 0,\mu ;x\right)  \notag \\
&&+2^{4\mu }\frac{\ln 4}{x^{\mu }}\mathrm{B}\left( \mu +\frac{1}{2},\mu +%
\frac{1}{2}\right) \,\,\Gamma \left( 1+\mu \right) I_{\mu }\left( \frac{x}{2}%
\right) .  \notag
\end{eqnarray}%
Now, insert in (\ref{J3(0,mu;x)_(1)})\ the result given in (\ref%
{J1(0,mu;x)_(1)})\ and simplify to get for $x>0$%
\begin{eqnarray}
&&J_{3}\left( 0,\mu ;x\right)  \label{J3(0,mu;x)_(2)} \\
&=&\mathrm{B}\left( \mu +\frac{1}{2},\mu +\frac{1}{2}\right) \Gamma \left(
1+\mu \right) \left( \frac{16}{x}\right) ^{\mu }  \notag \\
&&\left\{ I_{\mu }\left( \frac{x}{2}\right) \left[ \psi \left( \mu +\frac{1}{%
2}\right) +\ln \left( \frac{4}{x}\right) \right] +\frac{\partial I_{\mu
}\left( x/2\right) }{\partial \mu }\right\} .  \notag
\end{eqnarray}%
Finally, note that performing in (\ref{J3_def}) the change of variables $%
\tau =-t$, we obtain the reflection formula%
\begin{equation}
J_{3}\left( 0,\mu ;x\right) =J_{3}\left( 0,\mu ;-x\right) ,
\label{J3_reflection}
\end{equation}%
so that from (\ref{J3(0,mu;x)_(2)})\ and (\ref{J3_reflection}) we arrive at (%
\ref{J3(0,mu;x)_resultado}), as we wanted to prove.
\end{proof}

Table \ref{Table_3B} shows the integral $J_{1}\left( \kappa ,\mu ;x\right) $
for particular values of the parameters $\kappa $ and $\mu $, and $x\in
\mathbb{R}
$, obtained from (\ref{J1_general}), (\ref{J1_ml_resultado})\ and (\ref%
{J1(0,mu;x)_resultado}) with the aid of MATHEMATICA\ program.

\begin{center}
\begin{table}[htbp] \centering%
\caption{Integral $J_1(\kappa,\mu;x)$  for particular values of $\kappa$ and
$\mu$.}%
\rotatebox{90}{
\begin{tabular}{|c|c|c|}
\hline
$\kappa $ & $\mu $ & $J_{1}\left( \kappa ,\mu ;x\right) $ \\ \hline\hline
$-1$ & $0$ & $\pi \left\{ 2\,e^{x/2}\left( \ln 4-2\right) \left[ \left(
x+1\right) I_{0}\left( \frac{x}{2}\right) +x\,I_{1}\left( \frac{x}{2}\right) %
\right] -G^{\left( 1\right) }\left( \frac{3}{2};1;x\right) -2\,H^{\left(
1\right) }\left( \frac{3}{2};1;x\right) \right\} $ \\ \hline
$-\frac{1}{2}$ & $1$ & $x^{-2}\left\{ e^{x}\left[ \left( x-1\right) \left(
\mathrm{Chi}\left( x\right) -\mathrm{Shi}\left( x\right) -\ln x-\gamma
\right) -2\right] +\mathrm{Chi}\left( x\right) +\mathrm{Shi}\left( x\right)
-\ln x-\gamma +2\right\} $ \\ \hline
$-\frac{1}{3}$ & $0$ & $2\pi \left\{ G^{\left( 1\right) }\left( \frac{5}{6}%
;1;x\right) +2\,H^{\left( 1\right) }\left( \frac{5}{6};1;x\right) -\ln
\left( 432\right) L_{-5/6}\left( x\right) \right\} $ \\ \hline
$0$ & $0$ & $-\pi \,e^{x/2}\left\{ K_{0}\left( \frac{\left\vert x\right\vert
}{2}\right) +\left[ \ln \left( 4\left\vert x\right\vert \right) +\gamma %
\right] I_{0}\left( \frac{\left\vert x\right\vert }{2}\right) \right\} $ \\
\hline
$0$ & $\frac{1}{2}$ & $x^{-1}\left\{ e^{x}\left[ \mathrm{Chi}\left( x\right)
-\mathrm{Shi}\left( x\right) -\ln x-\gamma \right] -\mathrm{Chi}\left(
x\right) -\mathrm{Shi}\left( x\right) +\ln x+\gamma \right\} $ \\ \hline
$0$ & $1$ & $%
\begin{array}{l}
\left\{ I_{1}\left( \frac{\left\vert x\right\vert }{2}\right) \left[
I_{1}\left( \frac{\left\vert x\right\vert }{2}\right) K_{1}\left( \frac{%
\left\vert x\right\vert }{2}\right) -\ln \left( 4\left\vert x\right\vert
\right) -\gamma +2-\frac{1}{2\sqrt{\pi }}G_{1,3}^{2,1}\left( \frac{x^{2}}{4}%
;1/2;0,0,-1\right) \right] \right. \\
\quad +\left. K_{1}\left( \frac{\left\vert x\right\vert }{2}\right) \left[
1-I_{0}^{2}\left( \frac{\left\vert x\right\vert }{2}\right) \right] \right\}
\frac{\pi }{2\left\vert x\right\vert }e^{x/2}%
\end{array}%
$ \\ \hline
$\frac{1}{3}$ & $0$ & $2\pi \left\{ G^{\left( 1\right) }\left( \frac{1}{6}%
;1;x\right) +2\,H^{\left( 1\right) }\left( \frac{1}{6};1;x\right) -\ln
\left( 432\right) L_{-1/6}\left( x\right) \right\} $ \\ \hline
$\frac{1}{2}$ & $\frac{1}{2}$ & $\frac{\pi }{2}\left\{ G^{\left( 1\right)
}\left( \frac{1}{2};2;x\right) +2\,H^{\left( 1\right) }\left( \frac{1}{2}%
;2;x\right) -2\,e^{x/2}\ln 4\left[ I_{0}\left( \frac{x}{2}\right)
-\,I_{1}\left( \frac{x}{2}\right) \right] \right\} $ \\ \hline
$\frac{1}{2}$ & $1$ & $x^{-2}\left\{ e^{x}\left[ \mathrm{Chi}\left( x\right)
-\mathrm{Shi}\left( x\right) -\ln x-\gamma +2\right] -\left( x+1\right) %
\left[ \mathrm{Chi}\left( x\right) +\mathrm{Shi}\left( x\right) -\ln
x-\gamma \right] -2\right\} $ \\ \hline
\end{tabular}%
}
\label{Table_3B}%
\end{table}%
\end{center}

\section{Conclusions}

The Whittaker function $\mathrm{M}_{\kappa ,\mu }\left( x\right) $ is
defined in terms of the Kummer confluent hypergeometric function, hence its
derivative with respect to the parameters $\kappa $ and $\mu $ can be
expressed as infinite sums of quotients of the digamma and gamma functions.
Also, the parameter differentiation of the integral representations of $%
\mathrm{M}_{\kappa ,\mu }\left( x\right) $ leads to finite and infinite
integrals of elementary functions. These sums and integrals has been
calculated for particular values of the parameters $\kappa $ and $\mu $ in
closed-form. As an application of these results, we have obtained some
reduction formulas for the derivatives of the confluent Kummer function with
respect to the parameters, i.e. $G^{\left( 1\right) }\left( a,b;x\right) $
and $H^{\left( 1\right) }\left( a,b;x\right) $. Also, we have calculated
some finite integrals containing a combination of the exponential,
logarithmic and algebraic functions, as well as some infinite integrals
involving the exponential, logarithmic, algebraic and Bessel functions. It
is worth noting that all the results presented in this paper has been both
numerically and symbolically checked with MATHEMATICA program.

In the first Appendix, we have obtained the first derivative of the
incomplete gamma functions in closed-form. These results allow us to
calculate a finite logarithmic integral, which has been used to calculate
one of the integrals appearing in the body of the paper.

In the second Appendix, we have calculated some new reduction formulas for
the integral Whittaker functions $\mathrm{Mi}_{\kappa ,\mu }\left( x\right) $
and $\mathrm{mi}_{\kappa ,\mu }\left( x\right) $ from two reduction formulas
of the Whittaker function $\mathrm{M}_{\kappa ,\mu }\left( x\right) $. One
of the latter seems not to be reported in the literature.

In the third Appendix we collect some reduction formulas for the Whittaker
function $\mathrm{M}_{\kappa ,\mu }\left( x\right) $.

\appendix{}

\section{Parameter differentiation of the incomplete gamma functions}

\begin{definition}
The lower incomplete gamma function is defined as \cite{oldham2009atlas}:%
\begin{equation}
\gamma \left( \nu ,x\right) =\int_{0}^{x}t^{\nu -1}e^{-t}dt.
\label{gamma_def}
\end{equation}
\end{definition}

\begin{definition}
The upper incomplete gamma function is defined as \cite[Eqn. 45:3:2]%
{oldham2009atlas}%
\begin{equation}
\Gamma \left( \nu ,x\right) =\int_{x}^{\infty }t^{\nu -1}e^{-t}dt.
\label{Gamma_def}
\end{equation}
\end{definition}

The relation between both functions is
\begin{equation}
\Gamma \left( \nu \right) =\gamma \left( \nu ,x\right) +\Gamma \left( \nu
,x\right) .  \label{Gamma+gamma}
\end{equation}

The lower incomplete gamma function has the following series expansion \cite[%
Eqns. 45:6:1]{oldham2009atlas}
\begin{equation}
\gamma \left( \nu ,x\right) =e^{-x}\sum_{k=0}^{\infty }\frac{x^{k+\nu }}{%
\left( \nu \right) _{k+1}}.  \label{gamma_series_1}
\end{equation}

Also, the following integral representations in terms of infinite integrals
hold true \cite[Eqns. 8.6.3\&7]{olver2010nist} for $\mathrm{Re}\,z>0$,%
\begin{eqnarray*}
\gamma \left( \nu ,z\right) &=&z^{\nu }\int_{0}^{\infty }\exp \left( -\nu
t-z\,e^{-t}\right) dt, \\
\Gamma \left( \nu ,z\right) &=&z^{\nu }\int_{0}^{\infty }\exp \left( \nu
t-z\,e^{-t}\right) dt.
\end{eqnarray*}

From (\ref{gamma_def}), the derivative of the lower incomplete gamma
function with respect to the order $\nu $ has the following integral
representation:\
\begin{equation}
\frac{\partial \gamma \left( \nu ,x\right) }{\partial \nu }%
=\int_{0}^{x}t^{\nu -1}e^{-t}\ln t\,dt  \label{D_gamma_int}
\end{equation}

\begin{theorem}
The parameter derivative of the lower incomplete gamma function is%
\begin{equation}
\frac{\partial \gamma \left( \nu ,x\right) }{\partial \nu }=\gamma \left(
\nu ,x\right) \ln x-\frac{x^{\nu }}{\nu ^{2}}\,_{2}F_{2}\left( \left.
\begin{array}{c}
\nu ,\nu \\
\nu +1,\nu +1%
\end{array}%
\right\vert -x\right) .  \label{D_gamma}
\end{equation}
\end{theorem}

\begin{proof}
According to (\ref{gamma_def}) and (\ref{gamma_series_1}), the derivative of
the lower incomplete gamma function with respect to the parameter $\nu $ is,%
\begin{eqnarray*}
\frac{\partial \gamma \left( \nu ,x\right) }{\partial \nu }
&=&e^{-x}\sum_{k=0}^{\infty }\frac{x^{k+\nu }\left[ \ln x+\psi \left( \nu
\right) -\psi \left( k+1+\nu \right) \right] }{\left( \nu \right) _{k+1}} \\
&=&\left[ \ln x+\psi \left( \nu \right) \right] \gamma \left( \nu ,x\right)
-e^{-x}\sum_{k=0}^{\infty }\frac{x^{k+\nu -1}}{\left( \nu \right) _{k}}\psi
\left( k+\nu \right) .
\end{eqnarray*}%
Now, apply the sum formula \cite[Eqn. 6.2.1(63)]{brychkov2008handbook}
\begin{eqnarray*}
&&\sum_{k=0}^{\infty }\frac{t^{k}}{\left( a\right) _{k}}\psi \left(
k+a\right) \\
&=&\psi \left( a\right) +e^{t}\left[ t^{1-a}\psi \left( a\right) \,\gamma
\left( a,t\right) +\frac{t}{a^{2}}\,_{2}F_{2}\left( \left.
\begin{array}{c}
a,a \\
a+1,a+1%
\end{array}%
\right\vert -t\right) \right] ,
\end{eqnarray*}%
to arrive at (\ref{D_gamma}), as we wanted to prove.
\end{proof}

\begin{theorem}
The parameter derivative of the upper incomplete gamma function is%
\begin{eqnarray}
&&\frac{\partial \Gamma \left( \nu ,x\right) }{\partial \nu }
\label{D_Gamma} \\
&=&\Gamma \left( \nu \right) \psi \left( \nu \right) -\gamma \left( \nu
,x\right) \ln x+\frac{x^{\nu }}{\nu ^{2}}\,_{2}F_{2}\left( \left.
\begin{array}{c}
\nu ,\nu \\
\nu +1,\nu +1%
\end{array}%
\right\vert -x\right) .  \notag
\end{eqnarray}
\end{theorem}

\begin{proof}
Differentiate (\ref{Gamma+gamma})\ with respect to the parameter $\nu $ and
apply the result given in (\ref{D_gamma}).
\end{proof}

\begin{corollary}
From (\ref{D_gamma_int})\ and (\ref{D_gamma}), we calculate the following
integral:%
\begin{equation}
\int_{0}^{x}t^{\nu -1}e^{-t}\ln t\,dt=\gamma \left( \nu ,x\right) \ln x-%
\frac{x^{\nu }}{\nu ^{2}}\,_{2}F_{2}\left( \left.
\begin{array}{c}
\nu ,\nu \\
\nu +1,\nu +1%
\end{array}%
\right\vert -x\right) .  \label{gamma_int}
\end{equation}
\end{corollary}

\begin{corollary}
The following integral holds true for $x\in
\mathbb{R}
$:%
\begin{equation}
\int_{0}^{1}e^{xt}t^{\nu -1}\ln t\,dt=-\frac{1}{\nu ^{2}}\,_{2}F_{2}\left(
\left.
\begin{array}{c}
\nu ,\nu \\
\nu +1,\nu +1%
\end{array}%
\right\vert x\right) .  \label{gamma_int_2}
\end{equation}
\end{corollary}

\begin{proof}
Perform the change of variables $t=z\,\tau $ in the integral given in (\ref%
{gamma_int}), split the result in two integrals and apply again the change
of variables $t=x\,\tau $ to the first integral,%
\begin{eqnarray}
\int_{0}^{x}t^{\nu -1}e^{-t}\ln t\,dt &=&x^{\nu }\left[ \ln
x\int_{0}^{1}\tau ^{\nu -1}e^{-x\tau }\,d\tau +\int_{0}^{1}t^{\nu
-1}e^{-x\tau }\ln \tau \,d\tau \right]  \notag \\
&=&\ln x\underset{\gamma \left( \nu ,x\right) }{\underbrace{%
\int_{0}^{x}t^{\nu -1}e^{-t}\,dt}}+\,x^{\nu }\int_{0}^{1}\tau ^{\nu
-1}e^{-x\tau }\ln \tau \,d\tau .  \label{gamma_int_resultado}
\end{eqnarray}%
Comparing (\ref{gamma_int})\ to (\ref{gamma_int_resultado}), we obtain (\ref%
{gamma_int_2}), as we wanted to prove.
\end{proof}

\begin{corollary}
According to the notation given in (\ref{H(1)_def}), the following reduction
formula holds true for $x\in
\mathbb{R}
$:%
\begin{equation}
H^{\left( 1\right) }\left( \left.
\begin{array}{c}
1 \\
b%
\end{array}%
\right\vert x\right) =-\frac{x\,e^{x}}{b^{2}}\,_{2}F_{2}\left( \left.
\begin{array}{c}
b,b \\
b+1,b+1%
\end{array}%
\right\vert -x\right) .  \label{H(1)_(1,b)_resultado}
\end{equation}
\end{corollary}

\begin{proof}
Knowing that \cite[Eqn. 47:4:6]{oldham2009atlas}%
\begin{equation*}
_{1}F_{1}\left( \left.
\begin{array}{c}
1 \\
b%
\end{array}%
\right\vert z\right) =1+z^{1-b}e^{z}\gamma \left( b,z\right) ,
\end{equation*}%
and applying (\ref{D_gamma}), we calculate (\ref{H(1)_(1,b)_resultado}), as
we wanted to prove.
\end{proof}

\section{Reduction formulas for integral Whittaker\ functions $\mathrm{Mi}_{%
\protect\kappa ,\protect\mu }$ and $\mathrm{mi}_{\protect\kappa ,\protect\mu %
}$}

In \cite{apelblat2021integral}, we found some reduction formulas for the
integral Whittaker function $\mathrm{Mi}_{\kappa ,\mu }\left( x\right) $.
Next, we derive some new reduction formulas for $\mathrm{Mi}_{\kappa ,\mu
}\left( x\right) $ and $\mathrm{mi}_{\kappa ,\mu }\left( x\right) $ from
reduction formulas of the Whittaker function $\mathrm{M}_{\kappa ,\mu
}\left( x\right) $.

\begin{theorem}
The following reduction formula holds true for $x\in
\mathbb{R}
$, $n=0,1,2,\ldots $ and $\kappa >0$:%
\begin{equation}
\mathrm{Mi}_{\kappa +n,\kappa -1/2}\left( x\right) =2^{\kappa }\sum_{m=0}^{n}%
\binom{n}{m}\frac{\left( -2\right) ^{m}}{\left( 2\kappa \right) _{m}}\gamma
\left( \kappa +m,x/2\right) ,  \label{Mi_k+n,k-1/2}
\end{equation}%
where $\gamma \left( \nu ,z\right) $ denotes the lower incomplete gamma
function.
\end{theorem}

\begin{proof}
Apply to the definition of the Whittaker function (\ref{M_k,mu_def})\ the
reduction formula \cite[Eqn. 7.11.1(17)]{prudnikov1986integrals}%
\begin{equation*}
_{1}F_{1}\left( \left.
\begin{array}{c}
-n \\
b%
\end{array}%
\right\vert z\right) =\frac{n!}{\left( b\right) _{n}}L_{n}^{\left(
b-1\right) }\left( z\right) ,
\end{equation*}%
to obtain \cite[Eqn. 13.18.17]{olver2010nist}
\begin{equation}
\mathrm{M}_{\kappa +n,\kappa -1/2}\left( x\right) =\frac{n!\,e^{-x/2}x^{%
\kappa }}{\left( 2\kappa \right) _{n}}L_{n}^{\left( 2\kappa -1\right)
}\left( x\right) ,  \label{M_k+n,k-1/2}
\end{equation}%
where \cite[Eqn. 4.17.2]{lebedev1965special}%
\begin{equation}
L_{n}^{\left( \alpha \right) }\left( x\right) =\sum_{m=0}^{n}\frac{\Gamma
\left( n+\alpha +1\right) }{\Gamma \left( m+\alpha +1\right) }\frac{\left(
-x\right) ^{m}}{m!\left( n-m\right) !},  \label{Laguerre_def}
\end{equation}%
denotes the Laguerre polynomials. Insert (\ref{Laguerre_def})\ in (\ref%
{M_k+n,k-1/2})\ and integrate term by term according to the definition of
the integral Whittaker function (\ref{Mi_def}), to get%
\begin{eqnarray*}
&&\mathrm{Mi}_{\kappa +n,\kappa -1/2}\left( x\right) \\
&=&\sum_{m=0}^{n}\binom{n}{m}\frac{\left( -1\right) ^{m}}{\left( 2\kappa
\right) _{m}}\int_{0}^{x}e^{-t/2}t^{\kappa +m-1}dt.
\end{eqnarray*}%
Finally, take into account the defintion of the lower incomplete gamma
function (\ref{gamma_def}) and simplify the result to arrive at (\ref%
{Mi_k+n,k-1/2}), as we wanted to prove.
\end{proof}

\begin{remark}
Taking $n=0$ in (\ref{Mi_k+n,k-1/2}), we recover the formula given in \cite%
{apelblat2021integral}.
\end{remark}

\begin{theorem}
The following reduction formula holds true for $x>0$, $n=0,1,2,\ldots $ and $%
\kappa \in
\mathbb{R}
$:%
\begin{equation}
\mathrm{mi}_{\kappa +n,\kappa -1/2}\left( x\right) =2^{\kappa }\sum_{m=0}^{n}%
\binom{n}{m}\frac{\left( -2\right) ^{m}}{\left( 2\kappa \right) _{m}}\Gamma
\left( \kappa +m,x/2\right) ,  \label{mi_k+n,k-1/2}
\end{equation}%
where $\Gamma \left( \nu ,z\right) $ denotes the upper incomplete gamma
function.
\end{theorem}

\begin{proof}
Follow similar steps as in the previous theorem, but consider the definition
of the upper incomplete gamma function (\ref{Gamma_def}).
\end{proof}

\begin{theorem}
The following reduction formula holds true for $x\in
\mathbb{R}
$, $n=0,1,2,\ldots $ and $\kappa >0$:%
\begin{equation}
\mathrm{Mi}_{-\kappa -n,\kappa -1/2}\left( x\right) =\left( -1\right) ^{-%
\mathrm{sign}\left( x\right) \kappa }2^{\kappa }\sum_{m=0}^{n}\binom{n}{m}%
\frac{\left( -2\right) ^{m}}{\left( 2\kappa \right) _{m}}\gamma \left(
\kappa +m,-x/2\right) .  \label{Mi_-k-n,k-1/2}
\end{equation}
\end{theorem}

\begin{proof}
From\ the property for $x>0$ \cite[Eqn. 48:13:3]{oldham2009atlas}
\begin{equation*}
\mathrm{M}_{\kappa ,\mu }\left( -x\right) =\left( -1\right) ^{\mu +1/2}%
\mathrm{M}_{-\kappa ,\mu }\left( x\right) ,
\end{equation*}%
we have, for $x\in
\mathbb{R}
$,
\begin{equation}
\mathrm{M}_{-\kappa ,\mu }\left( x\right) =\left( -1\right) ^{-\mathrm{sign}%
\left( x\right) \left( \mu +1/2\right) }\mathrm{M}_{\kappa ,\mu }\left(
-x\right) ,  \label{Reflection_M}
\end{equation}%
Apply (\ref{Reflection_M})\ to (\ref{M_k+n,k-1/2})\ to obtain%
\begin{equation}
\mathrm{M}_{-\kappa -n,-\kappa -1/2}\left( x\right) =\left( -1\right) ^{-%
\mathrm{sign}\left( x\right) \,\kappa }\frac{n!\,e^{x/2}\left( -x\right)
^{\kappa }}{\left( 2\kappa \right) _{n}}L_{n}^{\left( 2\kappa -1\right)
}\left( -x\right) .  \label{M_-k-n,k-1/2}
\end{equation}%
Now, insert (\ref{Laguerre_def})\ in (\ref{M_k+n,k-1/2})\ and integrate term
by term according to the definition of the integral Whittaker function (\ref%
{Mi_def}), to get%
\begin{eqnarray*}
&&\mathrm{Mi}_{-\kappa -n,\kappa -1/2}\left( x\right) \\
&=&\left( -1\right) ^{-\mathrm{sign}\left( x\right) \,\kappa }\sum_{m=0}^{n}%
\frac{1}{\left( 2\kappa \right) _{m}}\binom{n}{m}\int_{0}^{x}e^{t/2}t^{m-1}%
\left( -t\right) ^{\kappa }dt.
\end{eqnarray*}%
Finally, take into account the defintion of the lower incomplete gamma
function (\ref{gamma_def}) and simplify the result to arrive at (\ref%
{Mi_-k-n,k-1/2}), as we wanted to prove.
\end{proof}

\begin{remark}
It is worth noting that we could not locate the reduction formula (\ref%
{M_-k-n,k-1/2})\ in the literature.
\end{remark}

\section{Reduction formulas for the Whittaker function $\mathrm{M}_{\protect%
\kappa ,\protect\mu }\left( x\right) $\label{Appendix_Reduction_Whittaker}}

For convenience of the readers, reduction formulas for the Whittaker
function $\mathrm{M}_{\kappa ,\mu }\left( x\right) $ are presented in their
explicit form in Table \ref{Table_5} for $x\in
\mathbb{R}
$.

\begin{center}
\begin{table}[tbp] \centering%
\caption{Whittaker function $M_{\kappa,\mu}(x)$  for particular values of $\kappa$ and
$\mu$.}%
\begin{tabular}{|c|c|c|}
\hline
$\kappa $ & $\mu $ & $\mathrm{M}_{\kappa ,\mu }\left( x\right) $ \\
\hline\hline
$-\frac{1}{4}$ & $\frac{1}{4}$ & $\frac{\sqrt{\pi }}{2}e^{x/2}x^{1/4}\mathrm{%
erf}\left( \sqrt{x}\right) $ \\ \hline
$-\frac{1}{2}$ & $\frac{1}{2}$ & $x\left[ I_{0}\left( \frac{x}{2}\right)
+I_{1}\left( \frac{x}{2}\right) \right] $ \\ \hline
$-\frac{1}{2}$ & $\frac{1}{6}$ & $2^{-2/3}x\,\Gamma \left( \frac{2}{3}%
\right) \left[ I_{-1/3}\left( \frac{x}{2}\right) +I_{2/3}\left( \frac{x}{2}%
\right) \right] $ \\ \hline
$-\frac{1}{2}$ & $1$ & $x^{-1/2}e^{-x/2}\left[ 2\,e^{x}\left( x-1\right) +2%
\right] $ \\ \hline
$0$ & $0$ & $\sqrt{x}\,I_{0}\left( \frac{x}{2}\right) $ \\ \hline
$0$ & $\frac{1}{2}$ & $2\,\sinh \left( \frac{x}{2}\right) $ \\ \hline
$0$ & $1$ & $4\sqrt{x}\,I_{1}\left( \frac{x}{2}\right) $ \\ \hline
$0$ & $\frac{3}{2}$ & $12\left[ \cosh \left( \frac{x}{2}\right) -\frac{2}{x}%
\sinh \left( \frac{x}{2}\right) \right] $ \\ \hline
$0$ & $\frac{5}{2}$ & $120\,x^{-2}\left[ \left( x^{2}+12\right) \sinh \left(
\frac{x}{2}\right) -6\,x\,\cosh \left( \frac{x}{2}\right) \right] $ \\ \hline
$\frac{1}{6}$ & $0$ & $\sqrt{x}e^{-x/2}L_{-1/3}\left( x\right) $ \\ \hline
$\frac{1}{4}$ & $-\frac{1}{4}$ & $x^{1/4}e^{-x/2}$ \\ \hline
$\frac{1}{4}$ & $\frac{1}{4}$ & $x^{1/4}e^{x/2}F\left( \sqrt{x}\right) $ \\
\hline
$\frac{1}{3}$ & $0$ & $\sqrt{x}e^{-x/2}L_{-1/6}\left( x\right) $ \\ \hline
$\frac{1}{2}$ & $\frac{1}{6}$ & $2^{-2/3}x\,\Gamma \left( \frac{2}{3}\right) %
\left[ I_{-1/3}\left( \frac{x}{2}\right) -I_{2/3}\left( \frac{x}{2}\right) %
\right] $ \\ \hline
$\frac{1}{2}$ & $\frac{1}{4}$ & $2^{-1/2}x\,\Gamma \left( \frac{3}{4}\right) %
\left[ I_{-1/4}\left( \frac{x}{2}\right) -I_{3/4}\left( \frac{x}{2}\right) %
\right] $ \\ \hline
$\frac{1}{2}$ & $\frac{1}{2}$ & $x\left[ I_{0}\left( \frac{x}{2}\right)
-I_{1}\left( \frac{x}{2}\right) \right] $ \\ \hline
$\frac{1}{2}$ & $1$ & $2\,x^{-1/2}e^{-x/2}\,\left( e^{x}-x-1\right) $ \\
\hline
$\frac{1}{2}$ & $2$ & $12\,x^{-3/2}e^{-x/2}\,\left[ 2\,e^{x}\left(
x-3\right) +x^{2}+4x+6\right] $ \\ \hline
$1$ & $-\frac{3}{2}$ & $e^{-x/2}\left( \frac{x}{2}+1+\frac{1}{x}\right) $ \\
\hline
$1$ & $1$ & $\frac{4}{3}\sqrt{x}\left[ x\,I_{0}\left( \frac{x}{2}\right)
-\left( x+1\right) I_{1}\left( \frac{x}{2}\right) \right] $ \\ \hline
$1$ & $\frac{3}{2}$ & $x^{-1}e^{-x/2}\,\left( 6\,e^{x}-3x^{2}-6\,x-6\right) $
\\ \hline
$1$ & $2$ & $\frac{32}{5}x^{-1/2}\left[ \left( x^{2}+4x+12\right)
I_{1}\left( \frac{x}{2}\right) -\left( x^{2}+3x\right) I_{0}\left( \frac{x}{2%
}\right) \right] $ \\ \hline
$2$ & $2$ & $\frac{32}{35}x^{-1/2}\left[ x\left( 2x^{2}+2x+3\right)
I_{0}\left( \frac{x}{2}\right) -2\left( x^{3}+2x^{2}+4x+6\right) I_{1}\left(
\frac{x}{2}\right) \right] $ \\ \hline
\end{tabular}%
\label{Table_5}%
\end{table}%
\end{center}

\bibliographystyle{plain}

\end{document}